\documentclass[10pt,twoside,english,reqno,a4paper]{amsart}

\usepackage[footskip=10mm]{geometry}
\usepackage{datetime}
\usepackage{fancyhdr}
\usepackage{amsmath}

\usepackage{listings,graphicx,varioref,color,bm,stmaryrd}

\usepackage{amsmath}
\usepackage[notref,notcite]{showkeys}

\usepackage{amsfonts}
\usepackage{calrsfs}

\usepackage{amssymb}
\usepackage{mathrsfs}
\usepackage{makecell}
\usepackage[mathscr]{euscript}

\usepackage{mathpazo} 
\usepackage{courier}

\usepackage{amscd}

\usepackage[T1]{fontenc}
\usepackage{graphics}
\usepackage[english]{babel}
\usepackage[utf8]{inputenc}
\usepackage[makeroom]{cancel}
\usepackage{pifont}
\usepackage{marvosym}
\usepackage{array}
\usepackage{datetime}
\usepackage[numbers]{natbib}

\definecolor{caribbeangreen}{rgb}{0.0, 0.8, 0.6}
\definecolor{darkpastelgreen}{rgb}{0.01, 0.75, 0.24}
\definecolor{green(pigment)}{rgb}{0.0, 0.65, 0.31}
\usepackage[dvipsnames]{xcolor}
\usepackage{hyperref}
\hypersetup{
colorlinks=true,
linkcolor= red!70!black,
citecolor= Orange,
urlcolor=Blue, 
}

\addto\extrasenglish{%
}

\addto\extrasenglish{%
}

\addto\extrasenglish{%
}

\usepackage[capitalize,nameinlink,noabbrev,nosort]{cleveref} 

\usepackage{tikz-cd}
\usetikzlibrary{patterns}
\usepackage{enumerate}
\usepackage{enumitem}
\usepackage{graphics}
\usepackage{pgfplots}
\pgfplotsset{width=7cm,compat=newest}
\usepackage{graphics}
\usepackage{tikz}
\usetikzlibrary{shapes.geometric}
\usepackage{tikz-cd}
\usepackage{caption}
\usepackage{subcaption}

\usepackage{float}

\newtheoremstyle{mytheoremstyle} 
{1em}                    
{1em}                    
{\slshape}                   
{}                           
{\bf}             
{.}                          
{.5em}                       
{}  

\theoremstyle{mytheoremstyle}
\newtheorem{theorem}{Theorem}[section]

\newtheorem{lemma}[theorem]{Lemma}

\newtheorem{remark}[theorem]{Remark}

\numberwithin{equation}{section}
\long\def\salta#1{\relax}

\usepackage{kantlipsum}
\allowdisplaybreaks

\usepackage[normalem]{ulem}               
\newcommand\redout{\bgroup\markoverwith
{\textcolor{bor}{\rule[0.5ex]{2pt}{0.8pt}}}\ULon}

\makeatletter

\def\og{\leavevmode\raise.3ex\hbox{$\scriptscriptstyle\langle\!\langle$~}}
\def\fg{\leavevmode\raise.3ex\hbox{~$\!\scriptscriptstyle\,\rangle\!\rangle$}}

\usepackage{scalerel,stackengine}
\stackMath
\makeatletter
\newcommand\reallywidehat[1]{%
\savestack{\tmpbox}{\stretchto{%
  \scaleto{%
    \scalerel*[\widthof{\ensuremath{#1}}]{\kern.1pt\mathchar"0362\kern.1pt}%
    {\rule{0ex}{\textheight}}
  }{\textheight}%
}{2.4ex}}%
\stackon[-6.9pt]{#1}{\tmpbox}%
}
\makeatother

\def\TT{{\mathbb{T}}}
\def\ZZ{{\mathbb{Z}}}
\def\pat{\partial_t}

\def\fg{\widehat{G}}
\def\tn{\mathcal{N}}

\newcommand{\Z}{\mathbb{Z}}

\newcommand{\Vast}{\bBigg@{3}}
\newcommand{\vast}{\bBigg@{2}}
\newcommand{\vvast}{\bBigg@{1.5}}

\def\de{\delta}
\def\eps{\varepsilon}

\def\D{\Delta }

\def\al{\alpha}
\def\de{\delta}

\newcommand{\N}{\nabla}

\newcommand{\NN}{\mathbb{N}}

\newcommand{\ds}{\displaystyle}

\DeclareMathOperator{\R}{\mathbb{R}}
\def\qq{\qquad}
\def\q{\quad}

\definecolor{bor}{cmyk}{0.21,0.93,0.86,0.12}
\definecolor{air}{rgb}{0.178, 0.51, 0.51}
\definecolor{range}{cmyk}{0,0.599,1,0.188}

\def\ds{\displaystyle}

\newcommand{\pare}[1]{\left(#1\right)}
\newcommand{\bra}[1]{\left[#1\right]}
\newcommand{\set}[1]{\left\{#1\right\}}
\newcommand{\av}[1]{\left|#1\right|}
\newcommand{\norm}[1]{\|#1\|}
\newcommand{\w}[1]{\widehat{#1}}

\renewcommand{\t}[1]{\text{#1}}

\newcommand{\m}[1]{\mathcal{#1}}

\def\cu{c_1}
\def\cd{c_2}

\usepackage[notcite,notref]{showkeys}
\usepackage{comment}
\usepackage{dsfont}

\newcommand{\tbf}{\textbf}
\newcommand{\tit}{\textit}

\newcommand{\tsl}{\textsl}

\newcommand{\mc}{\mathcal}
\newcommand{\mf}{\mathfrak}

\newcommand{\veps}{\varepsilon}
\newcommand{\what}{\widehat}

\newcommand{\oline}{\overline}

\renewcommand{\k}{\kappa}
\newcommand{\s}{\sigma}
\renewcommand{\t}{\tau}

\newcommand{\lan}{\langle}
\newcommand{\ran}{\rangle}
\newcommand{\T}{\mathbb{T}}

\renewcommand{\div}{{\rm div}\,}

\newcommand{\dx}{ \, {\rm d} x}
\newcommand{\dt}{ \, {\rm d} t}

\newcommand{\loc}{{\rm loc}}

\allowdisplaybreaks

\def\d{\partial}
\def\div{{\rm div}\,}
\newcommand{\dd}{{\rm d}}

\newcommand{\fra}[1]{\textcolor{blue}{FF: #1}}

\makeatletter
\@namedef{subjclassname@2020}{\textup{2020} Mathematics Subject Classification}
\makeatother

\theoremstyle{definition}
\newtheorem{example}[theorem]{Example}

\author[F. Fanelli]{Francesco Fanelli}

\author[R. Granero-Belinchón]{Rafael Granero-Belinchón}

\author[M. Magliocca]{Martina Magliocca}

\author[M. Pageard]{Matthieu Pageard}

\address[F. Fanelli]{Basque Center for Applied Mathematics, Alameda de Mazarredo 14, E-48009 Bilbao, Basque Country, Spain \\ and
Ikerbasque -- Basque Foundation for Science, Plaza Euskadi 5, E-48009 Bilbao, Basque Country, Spain \\ and
Universit\'e Claude Bernard Lyon 1, ICJ UMR5208, F-69622 Villeurbanne, France. 
 \url{ffanelli@bcamath.org}}

\address[R. Granero-Belinch\'on]{Departamento de Matem\'aticas, Estad\'istica y Computaci\'on, Universidad de Cantabria. Avda. Los Castros s/n, E-39005 Santander, Spain
 \url{rafael.granero@unican.es}}

\address[M. Magliocca]{Department of Mathematical Analysis,
Faculty of Mathematics, University of Sevilla,
C/Tarfia s/n, Campus Reina Mercedes,
Sevilla 41012, Spain.
 \url{mmagliocca@us.es}}

\address[M. Pageard]{Université Claude Bernard Lyon 1, ICJ UMR5208, F-69622 Villeurbanne Cedex, France \\
and Université de Toulouse, IMT UMR5219, F-31062 Toulouse Cedex 9, France.
 \url{pageard@math.univ-lyon1.fr}}



\usepackage{enumitem}
\setlist[itemize]{leftmargin=4mm}

\begin{document}

\title{Remarks on some quasi-linear fourth order parabolic equations related to the evolution of free surfaces}

\date\today 

\begin{abstract}
This note proves existence and uniqueness of strong solutions to a broad class of fourth-order quasi-linear parabolic equations in the class of
Wiener spaces. It improves upon previous results in the literature, devoted to some special cases falling within the general framework developed here,
in which uniqueness held true in a smaller class than the space where existence was proven.
Gain of analyticity, as well as generalisations to higher-order equations and to propagation of higher regularities, are also discussed.
\end{abstract}

\maketitle
\setcounter{tocdepth}{1}
\tableofcontents

\section{Introduction}

The goal of the present note is to establish a unified approach to study the well-posedness of scalar parabolic equations of the form
\begin{equation} \label{eq:parab}
 \left\{ \begin{array}{l}
          \d_tu\,+\,\alpha\,\Delta^2u\,=\,N(t,x,u,Du,\ldots,D^4u)\,, \\[1ex]
          u\bigl|_{t=0}\,=\,u_0\,,
         \end{array}
 \right.
\end{equation}
with $\al>0$, under suitable structural assumptions on the non-linearity $N$. Here, $t\in\R_+$ denotes the time variable and $x\in\Omega\subset\R^d$, with $d\geq1$,
the space variable. In order to simplify the study and really focus
on the structure of the non-linearity $N$, we assume the space domain to be the $d$-dimensional torus, that is
\[
\Omega\,=\,\T^d\,:=(\R/2\pi\Z)^d\,, 
\qquad\qquad \mbox{ with }\quad d\,\geq\,1\,.
\]

In system \eqref{eq:parab}, the notation $D^ku$, for $k\in\NN$, stands for the vector of space derivatives of order $k$, namely terms of the form $\d^\beta_x u$,
for multi-indices $\beta\in\NN^d$ such that $|\beta|=k$. For first-order derivatives,
we will alternatively use the notation $Du$ and $\nabla u$, pointing out the variable of differentiation,
\tsl{e.g.} $\nabla_x$, when needed.

Equations of the same kind as \eqref{eq:parab} typically describe the evolution of free surfaces. They naturally arise in many fields of applied mathematics, ranging from
crystal surfaces \citep{krug1995adatom, shehadeh2011evolution}
to thin films appearing in wetting-dewetting processes \cite{khenner2018modeling} and
epitaxial growth models \citep{escudero2023explicit, escudero2015global}, for instance.
We refer \tsl{e.g.} to \cite{granero2018global, magliocca2024fourth, escudero2015global, granero2024global, escudero2013radial}
and references therein for some rigorous mathematical studies of those models.
We also refer to  \tsl{e.g.} \cite{Gall-Moussa} and references therein for recent advances on quasi-linear parabolic equations and systems in which
the diffusion matrix is not uniformly elliptic.

\medbreak
To better explain the main motivations behind this work, let us specify the mathematical setting we will adopt throughout the paper.

In order to treat the non-linearity appearing on its right-hand side at a rather broad level of generality, and following the approach of
\cite{granero2018global, magliocca2024fourth, granero2024global}, we study the well-posedness of \eqref{eq:parab}
in critical spaces entering the class of Wiener spaces. More precisely, for any real number $s\geq0$, we define the (homogeneous) Wiener space $A^s(\TT^d)$ as
\begin{align*}\label{Wienerhomo}
{{A}}^s(\TT^d)\, &:=\,
\bigg\{f\,\in\, L^1(\TT^d)\quad \Big|\quad  \|f\|_{A^s}\,:=\,\sum_{k\in\ZZ^d\setminus\{0\}} |k|^s\;|\widehat{f}(k)|\,<\,+\,\infty\;\bigg\}\,.
\end{align*}
Notice that 
the quantity $\|f\|_{A^s}$ defined above is a homogeneous semi-norm, in the sense that it does not provide any information on
the low-frequency term $\what{f}(0)$, or equivalently on the spatial average $\lan f\ran $ of $f$ over $\TT^d$.
Therefore, 
it is convenient to introduce the subspace of $A^s(\TT^d)$ composed of mean-free functions, namely 
\[
A^s_0(\TT^d)\, :=\, \left\{f\,\in\, A^s(\TT^d)\quad \Big|\quad \lan f\ran\,:=\,\frac{1}{|\TT^d|}\,\int_{\TT^d}f(x)\,\dd x\,=\,0\right\}\,.
\]
Observe that the semi-norm $\|\cdot\|_{A^s}$ becomes a norm when acting on the subspace $A^s_0(\TT^d)$.
In particular, the space $A^0_0$ is a Banach space; moreover, it satisfies the embedding property
$A^0_0\hookrightarrow L^\infty$.

Then, given an initial datum
\begin{equation} \label{hyp:in-datum}
u_0\,\in\,A^0_0(\TT^d) \,,
\end{equation}
we will look for solutions to the initial value problem \eqref{eq:parab} in the space
\begin{equation} \label{def:X_{0,4}}
X_{0,4}\,:=\, \mc C_b\big(\R_+;A^0_0(\TT^d)\big)\,\cap\,L^1\big(\R_+;A^4(\TT^d)\big)\,,
\end{equation}
where we have set $\mc C_b\,:=\,\mc C\,\cap\,L^\infty$.

The main result of this paper, rigorously formulated in Section \ref{s:results} below,
can be roughly stated in the following way:


\tit{for a large class of non-linearities $N$, given an initial datum $u_0\in A^0_0(\T^d)$ and such that
$\|u_0\|_{A^0}$ is small enough, the Cauchy problem \eqref{eq:parab} is (globally in time) well-posed in the space $X_{0,4}$}.

\medbreak

Results of this kind have already been proved in \tsl{e.g.} \cite{granero2018global, magliocca2024fourth, granero2024global}
for the specific models under consideration, with
however a few important differences, which we now want to comment.
The proof of uniqueness in the space $X_{0,4}$ usually relies on stability estimates in low-regularity norms (for instance, in $A^0$) and requires to bound
the difference of the non-linear terms in those norms. This of course depends on the structure of the non-linear terms: we identify here somehow minimal
assumptions to be imposed on $N$ to carry out such stability estimates.
On the side of existence, instead, the classical proof typically employs a Galerkin approximation procedure, which allows to construct
a sequence of global smooth solutions. After getting uniform bounds in the space $X_{0,4}$, a solution to the specific problem at hand is then constructed
by a weak compactness method, plus, of course, an Ascoli-Arzel\`a or Aubin-Lions argument to get strong convergence and, in this way, pass
to the limit in the non-linear terms.

The subtle (but important) point here is that the compactness argument is typically carried out in classical Sobolev spaces. However,
as is well known, $L^1$ does not possess good properties in terms of weak-$*$ convergence; as a consequence,
the just depicted argument is able to produce solutions belonging only to
$Y_{0,4}\,:=\,L^\infty\big(\!\R_+;L^\infty(\T^d)\big)\,\cap\,\mc M\big(\!\R_+;W^{4,\infty}(\T^d)\big)$, where we have denoted by $\mc M\big(\!\R_+;Z\big)$ the space
of Radon measures over $\R_+$ with values in the Banach space $Z$. Observe that $X_{0,4}\,\subset\,Y_{0,4}$ with \emph{strict} inclusion.
Thus, existence is not obtained in the uniqueness class, which coincides with $X_{0,4}$,
but in the larger space $Y_{0,4}$.

Apart from proposing a unified approach for most of the equations previously considered, see \tsl{e.g.}
\cite{granero2018global, magliocca2024fourth, granero2024global}, the present work complements the aforementioned studies by upgrading
the regularity of the constructed solutions, which are then proved to belong to the uniqueness class.

\medbreak
The paper is organised in the following way. In Section \ref{s:results}, we present the precise assumptions on the non-linearity $N$ and state
our main results, about well-posedness and stability of solutions to the Cauchy problem \eqref{eq:parab}. Their proof is carried out
in Section \ref{s:proof-unique} (proof of the stability result, which in particular implies uniqueness) and Section \ref{s:proof-exist} (proof
of existence).
Section \ref{s:general} presents generalisations of our approach, mainly in three directions: firstly, we prove global propagation of higher regularities in case
of smoother initial data; moreover, we show gain of analyticity of solutions globally in time; finally, we discuss the adaptations to be carried out in our study
in order to treat the case of higher-order equations.
Finally, in Section \ref{s:applic} we show how the equations considered in previous studies
enter our framework, thus implying that our results apply in those cases.

\section{Main results} \label{s:results}

In this section, we collect our main assumptions, specifically related to the non-linearity $N$ appearing on the right-hand side of equation \eqref{eq:parab},
and we formulate our main results.

\subsection{Assumptions} \label{ss:assumpt}
Let us collect the structural assumptions on the non-linear term
\[
 N\,=\,N(t,x,p_0,\mf p)\,,\qquad \qquad \mf p\,:=\,(p_1,p_2,p_3,p_4)\,,
\]
that will be needed in our study. Recall that we consider $(t,x)\,\in\,\R_+\!\times\,\T^d$, with $d\geq1$. 
Here above, the variable $p_j$ $(j=0,\ldots,4)$ plays the role of the term $D^ju$; in particular we have $p_j\in \R^{d_j}$, with $d_j:=d^j$,
and, after denoting $\mf D:=\sum_{j=1}^4d_j$, we have $\mf p\in \R^{\mf D}$.
In what follows, we denote our assumptions with the label \tbf{(An)}, n$\,\in\NN$.
We postpone any remark at the end of this part, after stating all the hypotheses one after the other.


%

In the following, the parameter $\kappa\in\,]0,+\infty]$ denotes a fixed positive constant. It is needed to include possible singular or unbounded behaviour
of the non-linearity $N$ with respect to the variable $p_0$, which is a typical case in applications (see Section \ref{s:applic}). We refer to Remark \ref{rem:p0}
for more explanations about this.

\medbreak
\tbf{(A1)}
For any fixed $(p_0,\mf p)\,\in\;]-\k,\k[\,\times\R^{\mf D}$,
the map $(t,x)\,\mapsto\,N(t,x,p_0,\mf p)$ belongs to the space $L^\infty\big(\!\R_+;A^0(\T^d)\big)$.
In addition, there exist exponents $\ell_1,\ell_2,\ell_3\in\NN$ and $\ell_4>0$ such that,
for a.e. $t\in\R_+$ and for all $v\in A^4(\T^d)$ satisfying $\|v\|_{A^0}<\k$, one has
\begin{align}
\label{eq:disN}
\left\|N(t,\cdot,v,Dv,\ldots,D^4v)\right\|_{A^0}\,
\lesssim\,
\sum_{j=1}^{4}\left\|v\right\|_{A^0}^{\ell_j}\,\left\|D^jv\right\|_{A^0}^{4/j}\,.
\end{align}

\medbreak
\tbf{(A2)} For a.e. $(t,x)\in\R_+\times\T^d$, the function $N$ is \emph{locally Lipschitz continuous} with respect to the $(p_0,\mf p)$-variables
over $\,]-\k,\k[\,\times\R^{\mf D}$. In addition, there exist
exponents  $q_0,q_1,q_2,q_3\in\NN$ and $q_4>0$ such that,
for a.e. $t\in\R_+$ and for all $v\in A^4(\T^d)$ satisfying $\|v\|_{A^0}<\k$, one has
\begin{equation} \label{hyp:Lip-N}
\left\|\nabla_{p_j}N(t,\cdot,v, Dv,\ldots,D^4v)\right\|_{A^0}\,\lesssim\,\left\|v\right\|_{A^0}^{q_j}\,
\sum_{k=1}^4\left\|D^kv\right\|_{A^0}^{(4-j)/k} \quad \mbox{ for }\q 0\le j\le 4\,.
\end{equation}

\medbreak
\tbf{(A3)} There exists a function $H\,=\, H(t,x,p_0,p_1,p_2,p_3)$ such that, for a.e. $(t,x)\in\R_+\!\times\,\T^d$ and any
$v\in W^{4,\infty}(\T^d)$, one has
\[
 N(t,x,v,Dv,\ldots,D^4v)\,=\,\div H(t,x,v,Dv,D^2v,D^3v)\,,
\]
where the divergence operator acts only with respect to the space variable $x\in\T^d$.

\medbreak
Let us formulate some remarks about the previous assumptions. We start by commenting on hypotheses \tbf{(A1)} and \tbf{(A2)}.
We first comment on the necessity to take $p_0\in\,]-\kappa,\kappa[\,$ instead of $p_0\in\R$ in assumptions \tbf{(A1)} and \tbf{(A2)}.

\begin{remark}\label{rem:p0}
    Some of the non-linearities that we are going to deal with in Section \ref{s:applic} present either unbounded, or singular behaviour
    with respect to the variable $p_0$. For instance, in Subsections \ref{s:crystal} and \ref{ss:two-op}, estimating quantities such as
    $N(p_0,\mf p)$ or $\N_{p_j}N(p_0,\mf p)$, with $j=0,\ldots,4$, gives rise to terms of the form 
    \[
    e^{|p_0|} \qq \mbox{ or } \qq \frac{1}{(1-|p_0|)^\delta}\,,
    \]
    for some constant $\delta>0$. These terms obviously do not satisfy {\rm \tbf{(A1)}} and {\rm \tbf{(A2)}} globally in $p_0$, that is for $p_0\in\R$.
    This is the main reason for considering the restriction $p_0\in\,]-\kappa,\kappa[\,$, for some $\k\in\,]0,+\infty]$.
    At the same time, the variable $p_0$ playing the role of the scalar field $v$, we will need to check throughout the proof that $\|v(t)\|_{A^0}<\k$.
    Notice that, owing to the embedding $A^0_0(\TT^d)\,\hookrightarrow\, \mc C(\TT^d)$, this conditions also implies $|v(t,x)|<\k$, which makes
    $N(t,x,v,Dv,\ldots,D^4v)$ well-defined.
    
    Notice that, in the case where $\kappa<+\infty$, the terms involving $\|v\|_{A^0}$ on the right-hand side of inequality \eqref{eq:disN} could be
    bounded by some constant depending on $\kappa$, hence absorbed into the (implicit) multiplicative constant. However, for our proof to work, it
    is crucial that the term $\|D^4v\|_{A^0}$ is multiplied by a strictly positive power of $\|v\|_{A^0}$. Therefore, we decided to keep the explicit
    dependence of the right-hand side of \eqref{eq:disN} on the $A^0$ norm of $v$. The same remark applies to inequality \eqref{hyp:Lip-N}, where
    the presence of the term $\|v\|_{A^0}^{q_4}$, with $q_4>0$, is essential in our study. We refer to Remark \ref{rmk:p4} below for more details about this point.
\end{remark}

\begin{remark}\label{rmk:7}
It is not hard to check that, by our method, we can replace \eqref{eq:disN} by the more general assumption
\[
\left\|N(t,\cdot,v,Dv,\ldots,D^4v)\right\|_{A^0}\,
\lesssim\,
\left(
\left(1+\left\|v\right\|_{A^0}\right)^{\ell}\,\sum_{j=1}^{3}\left\|D^jv\right\|_{A^0}^{4/j}\,+\,\left\|v\right\|_{A^0}^{\ell_4}\,\left\|D^4v\right\|_{A^0}\,
\right),
\]
for some $\ell\in\NN$ and $\ell_4>0$, and a similar modification of \eqref{hyp:Lip-N} for the derivatives of $N$ with respect to the variables $p_j$.
\end{remark}

\begin{remark}\label{rmk:p4}
Estimate \eqref{eq:disN} is somehow ``linear'' with respect to $\|D^4v\|_{A^0}\equiv\|v\|_{A^4}$; in addition, this term
is only multiplied by a suitable power of $\|v\|_{A^0}$, but not by factors depending on the other quantities $Dv$, $D^2v$ and $D^3v$.
For consistency, this dependence reflects also in the formulation of assumption \eqref{hyp:Lip-N} and the dependence
of those estimates with respect to the $\|D^4v\|_{A^0}$.
Such special behaviour of \eqref{eq:disN} and \eqref{hyp:Lip-N} with respect to $\|D^4v\|_{A^0}$ is absolutely crucial for our method to work.
This is of course linked with the fact that we only deal here with quasi-linear problems.

In this respect we observe that, in particular, in the expression of $N$ we can allow
the presence of terms having the following behaviour:
\[
\left|p_0\right|^{\ell_j}\,\left|p_j\right|^{a_j}\,\left|p_4\right|^{b_j}\qquad\qquad \mbox{ for }\quad j\,\in\,\big\{1,2,3\big\}\,,
\]
for suitable exponents $\ell_j\in\NN$, $b_j\in[0,1[$ and $a_j\geq0$, under the condition that $a_j/(1-b_j)\leq 4/j$.
Indeed, under these assumptions,
it is easy to see that terms of that kind can be controlled (when taking $p_j=D^jv$ and passing to the $A^0$ norm)
by the right-hand side of \eqref{eq:disN}.
Further generalisations, with fourth-order (or higher) products of $p_j$'s are also possible, under suitable restrictions on the powers of the various terms.
\end{remark}

We now briefly comment on assumption \tbf{(A3)}. Its role is simply to ensure the persistence of the average of the initial datum, thus providing
a control on the low-frequency of the solution, which would otherwise be out of reach within our framework.
Observe that this assumption
will not be needed for the estimates of the $A^s$ norm of the solution, for $s\geq0$. For instance,
it is not useful to use the divergence form of $N$ and integrate by parts, in order to carry out those estimates.

\medbreak

Finally, we conclude this part by making an explicit example of non-linearity $N$ which satisfies our assumptions {\rm \tbf{(A1)}} and {\rm \tbf{(A2)}}.

\begin{example} \label{example}
Let us focus on the one-dimensional case $d=1$. Take the function $N$ independent of $(t,x)$ and defined by the formula below:
\[
N(p_0,\mf p)\,=\,\sum_{\substack{ 0\le i,j,s,q\le4\\ i+j+s+q=4}} p_0^{\ell_{i,j,s,q}}\, p_i\,p_j\,p_s\,p_q\,.
\]
Then, we claim that such $N$ satisfies both assumptions {\rm \tbf{(A1)}} and {\rm \tbf{(A2)}}.
Let us just focus on showing {\rm \tbf{(A1)}}, the proof of {\rm \tbf{(A2)}} being analogous.
Notice that, up to substituting $(p_0,\mf p)$ by $(v,Dv,\ldots,D^4v)$, passing to the $A^0$ norms on both sides of the equality
and using the property of algebra of the space $A^0$, it is enough to exhibit a bound for $N(p_0,\mf p)$ in terms of suitable powers of the elements of
the vector $(p_0,\mf p)$ itself.
With this remark in mind,  we can estimate
\begin{align*}
\av{N(p_0,\mf p)}\,&\lesssim\,\sum_{\substack{ 0\le i,j,s,q\le4\\ i+j+s+q=4}} \av{p_0}^{\ell_{i,j,s,q}}\, \av{p_i}\, \av{p_j} \,\av{p_s}\, \av{p_q}\\
&\lesssim\,\av{p_0}^{\ell_0}\,\av{p_4}\,+\, \av{p_0}^{\ell_1}\,\av{p_1}\,\av{p_3}\,+\,
\av{p_0}^{\ell_2}\,\av{p_1}^2\,\av{p_2}\,+\, \av{p_0}^{\ell_3}\,\av{p_1}^4\,+\, \av{p_0}^{\ell_4}\,\av{p_2}^2\,,
\end{align*}
which is easily reconducted to the general form \eqref{eq:disN} after an application of the Young inequality.
\end{example}

\subsection{Statement of the main results} \label{ss:statem}
Let us finally present the main results of this work, related to the well-posedness of the Cauchy problem \eqref{eq:parab}.

First of all,
we notice that, under assumption \tbf{(A3)}, the mean value of the initial datum is preserved: for any initial datum
$u_0\in L^\infty(\TT^d)$ and any (say) local-in-time solution $u\in L^\infty\big([0,T]\times\TT^d\big)\cap L^1\big([0,T];W^{4,\infty}(\TT^d)\big)$
to equation \eqref{eq:parab}, for some $T>0$, one has
\begin{equation} \label{est:average}
\forall\,t\in[0,T]\,,\qquad\qquad \lan u(t)\ran\,=\,\lan u_0\ran\,.
\end{equation}
This can be checked by integrating \eqref{eq:parab} over $\TT^d$ and using the space periodicity of $u(t)$ for any $t\ge0$. Therefore, we can safely restrict our attention to the case of mean-free initial data: from now on, we will always assume \eqref{hyp:in-datum}, that is
\[
 \lan u_0\ran\,=\,0
\]
and we will look for solutions being mean free as well. In particular, for this class of solutions, one has the continuous embedding
$A^0_0(\TT^d)\hookrightarrow {\mc C}(\TT^d)$.

As explained in the Introduction, we look for strong solutions of equation \eqref{eq:parab} in the functional framework of Wiener spaces.
Thus, for an initial datum $u_0$ as in \eqref{hyp:in-datum}, we aim to get a unique solution $u$ in the space
$X_{0,4}$ defined in \eqref{def:X_{0,4}}.

Let us remark that, in view of the interpolation inequality (see \emph{e.g.} \cite[Lemma 2.1]{bruell2019thin})
\begin{equation}\label{interpol}
\|f\|_{A^\s}\,\le\,\|f\|_{A^0}^{1-\theta}\,\|f\|_{A^s}^{\theta}\,,\qquad\qquad \mbox{ for }\qquad s>0\,, \quad 0\le \s\le s\,, \quad \theta\,=\,\frac{\s}{s}\,,
\end{equation}
we gather that, if $u\in X_{0,4}$, then we have 
\begin{equation}\label{eq:interpol_X0}
u\,\in\, L^4\big(\R_+;A^1(\TT^d)\big)\,\cap\,L^2\big(\R_+;A^2(\TT^d)\big)\,\cap\,L^{4/3}\big(\R_+;A^3(\TT^d)\big)\,.
\end{equation}
Under assumptions \tbf{(A1)} and \tbf{(A3)}, this implies that
\[
N\big(t,x,u(t,x),Du(t,x),\ldots,D^4u(t,x)\big)\,\in\,L^1\big(\R_+;A^0_0(\TT^d)\big)
\]
and, by use of the equation, one discovers that also $\partial_tu$ belongs to this space.

The consequence of this fact is twofold.
On the one hand, we have that the non-linearity and all the terms in \eqref{eq:parab} make sense pointwise, thus $u$ is a strong solution to that equation,
as claimed. On the other hand, owing to the continuous embedding $A^0_0(\TT^d)\hookrightarrow \mc C(\TT^d)$,
we deduce that $u\in\mc C\big(\R_+\!\times\TT^d\big)$, so taking its trace at time $t=0$ makes sense and the initial datum is attained in the sense of
distributions (and actually, even in a stronger sense).


After these preliminaries, we are now ready to state the main results of this paper.
We start with an existence and uniqueness statement in the class of Wiener spaces.

\begin{theorem}[Well-posedness]\label{teo:wp}
Consider the Cauchy problem \eqref{eq:parab}, under assumptions {\rm \tbf{(A1)}}, {\rm \tbf{(A2)}} and  {\rm \tbf{(A3)}} 
on the non-linearity $N$.

Then, there exists $\veps_0>0$ such that, for any $u_0\in A^0_0(\T^d)$ satisfying
$\left\|u_0\right\|_{A^0}\,\leq\,\veps_0$,
there exists a unique solution $u\in X_{0,4}$ to \eqref{eq:parab} related to this initial datum.
In addition, there exists a constant $C\,=\,C(\alpha,\veps_0)\,>\,0$ depending only on the diffusion coefficient
$\alpha$ and on $\veps_0$, such that
\begin{equation} \label{est:solution}
 \|u\|_{X_{0,4}}\,\leq\,C\,\left\|u_0\right\|_{A^0}\,.
\end{equation}
\end{theorem}

In order to establish the uniqueness part of the previous theorem, a key point is the use of the following stability
result.

\begin{theorem}[Stability] \label{t:stab}
Consider the Cauchy problem \eqref{eq:parab}, under assumptions {\rm \tbf{(A1)}}, {\rm \tbf{(A2)}} and {\rm \tbf{(A3)}}
on the non-linearity $N$. Let $u_0$ and $v_0$ be two initial data in $A^0_0(\T^d)$, and let respectively $u$ and $v$ be two
corresponding solutions belonging to the space $X_{0,4}$.
Suppose that, for suitable (implicit) universal multiplicative constants, one has the bounds
\begin{equation}\label{eq:t0-stab}
\norm{u}_{L^\infty(\R_+;A^0)}\,\lesssim\,\norm{u_0}_{A^0}\qq \mbox{ and } \qq\norm{v}_{L^\infty(\R_+;A^0)}\,\lesssim\,\norm{v_0}_{A^0}\,.
\end{equation}

Then, there exists $\veps_1>0$ such that, if $\norm{u_0}_{A^0}\,\leq\,\veps_1$ and $\norm{v_0}_{A^0}\,\leq\,\veps_1$, one has
\begin{align}
\|u-v\|_{X_{0,4}}
\,&\lesssim\,\norm{u_0-v_0}_{A^0}\,, \label{eq:dis-stab0}
\end{align}
where the implicit constant may depend on $\alpha$, $\veps_1$ and on the norms of $u$ and $v$ in $X_{0,4}$.
\end{theorem}

\section{Stability and uniqueness} \label{s:proof-unique}

This section is devoted to the proof of Theorem \ref{t:stab}. Notice that this stability result, combined with estimate \eqref{est:solution}
and with the preservation of the spatial average of the solution (see also relation \eqref{est:average-diff} below),
immediately implies the uniqueness statement of Theorem \ref{teo:wp}.

\begin{proof}[Proof of Theorem \ref{t:stab}]
Without lack of any generality, we detail the proof in the case of space dimension $d=1$, to simplify the notation.

We consider problem \eqref{eq:parab} with initial data $u\bigl|_{t=0}\,=u_0$ and $v\bigl|_{t=0}\,=v_0$.
Let us denote by $\de:=u-v$ the difference of the respective solutions, and let us denote for simplicity
\[
N[w](t,x) := N\big(t,x,w(t,x),Dw(t,x),\ldots,D^4w(t,x)\big)\qq \text{ for }\q w=u,v.
\]
Then, the difference $\de$ solves the equation
\begin{equation}
    \label{eq:diffdelta}
    \ds\pat \de\,+\,\al \D^2\de\,=\,N[u]-N[v]
\end{equation}
related to the initial datum $\de\bigl|_{t=0}\,=\,u_0-v_0$. Owing to assumption \tbf{(A3)}, we notice that
\begin{equation} \label{est:average-diff}
\forall\,t\in\R_+\,,\qquad\qquad \lan \de(t)\ran\,=\,\lan u_0-v_0\ran\,=\,0\,.
\end{equation}
Next, for $k\in\Z^d$, we use the symbol $\mf F[u](t,k)\,=\,\what u(t,k)$ to denote the $k$-th Fourier coefficient of the function $u=u(t,x)$
with respect to the space variable.
Taking the Fourier tranform of \eqref{eq:diffdelta}, we infer the equation
\[
\pat\w\de(t,k)\, =\, -\al|k|^4\w\de(t,k) + \mf F\big[N[u]-N[v]\big](t,k)\,.
\]
Using the formula
\begin{equation}\label{eq:pat}
\pat|\w{\de}(t,k)|={\mathrm{Re}\Big( \overline{\w{\de}}(t,k)\pat \w{\de}(t,k) \Big)}/{|\w{\de}(t,k)|}
\end{equation}
together with the fact that $\mathrm{Re}\big(\bar ab\big)\,\le\,|a|\,|b|$, we infer that 
\[
\pat|\w{\de}(t,k)| + \al|k|^4|\w\de(t,k)| \le \Big|\mf F\big[N[u]-N[v]\big](t,k)\Big|.
\]
Summing over $k\in\ZZ^d\setminus\{0\}$, we can finally estimate the $A^0$ norm of the difference $\de(t)$ as
\begin{align}\label{eq:diff}
\frac{\dd}{\dd t}\norm{\de(t)}_{A^0}\,+\,\al\norm{\de(t)}_{A^4}\,&\leq\,
\norm{N[u](t)-N[v](t)}_{A^0}\,.
\end{align}
Thus, in order to complete the proof, we need to bound the right-hand side of \eqref{eq:diff}.

To begin with, let us denote $\delta_\theta:=v+\theta(u-v)$ for $\theta\in[0,1]$. We notice that, by Taylor's formula, we can write
\begin{equation}
    \label{estim:nunv1}
    N[u]-N[v] = \sum_{j=0}^4 \,\int_0^1 \nabla_{p_j}N[\delta_\theta]\cdot D^j\delta \;\dd\theta.
\end{equation}
Thanks to \eqref{eq:t0-stab}, for $\veps_1>0$ small enough we can apply assumption \tbf{(A2)}: using also the Banach algebra property of $A^0$, we get
\[
\|\nabla_{p_j}N[\delta_\theta](t)\cdot D^j\delta(t)\|_{A^0}
\lesssim \|\delta_\theta(t)\|_{A^0}^{q_j}\left(\,\sum_{k=1}^4 \|\delta_\theta(t)\|_{A^k}^{(4-j)/k}\right)\|\delta(t)\|_{A^j}\,.
\]
On the one hand, for $j=0$ we have 
\begin{equation}
    \label{estim:p0}
    \|\nabla_{p_0}N[\delta_\theta](t)\cdot \delta(t)\|_{A^0}
\lesssim \|\delta_\theta(t)\|_{A^0}^{q_0}\left(\,\sum_{k=1}^4 \|\delta_\theta(t)\|_{A^k}^{4/k}\right)\|\delta(t)\|_{A^0},
\end{equation}
and for $j=4$,
\begin{equation}
    \label{estim:p4}
    \|\nabla_{p_4}N[\delta_\theta](t)\cdot D^4\delta(t)\|_{A^0}
\lesssim \|\delta_\theta(t)\|_{A^0}^{q_4}\|\delta(t)\|_{A^4}.
\end{equation}
We now focus on the indices $j=1,2,3$. Using, for $0\leq j\leq 4$, the interpolation inequality
\begin{equation}
    \label{interp-ex}
    \|f\|_{A^j} \le \|f\|_{A^0}^{1-j/4}\, \|f\|_{A^4}^{j/4}\,,
\end{equation}
we immediately see that one has
\[
\|\nabla_{p_j}N[\delta_\theta](t)\cdot D^j\delta(t)\|_{A^0}
\lesssim \|\delta_\theta(t)\|_{A^0}^{q_j}\left(\,\sum_{k=1}^4 \|\delta_\theta(t)\|_{A^k}^{(4-j)/k}\right)\|\delta(t)\|_{A^0}^{(4-j)/4} \|\delta(t)\|_{A^4}^{j/4}.
\]
Using now the Young inequality, we can write, for any $j=1,2,3$ and $\eta>0$, 
\begin{equation}
    \label{estim:pj}
    \|\nabla_{p_j}N[\delta_\theta](t)\cdot D^j\delta(t)\|_{A^0}
\lesssim \eta \|\delta_\theta(t)\|_{A^0}^{4q_j/j} \|\delta(t)\|_{A^4} + C_j(\eta) \left(\,\sum_{k=1}^4 \|\delta_\theta(t)\|_{A^k}^{4/k}\right)\|\delta(t)\|_{A^0},
\end{equation}
for some constant $C_j(\eta)>0$ such that $C_j(\eta)\to+\infty$ as $\eta\to0$, where we have also used the discrete Jensen inequality to absorb the exponent $4/(4-j)$
inside the sum. 

At this point, we regroup the terms proportional to $\|\delta(t)\|_{A^4}$ and those proportional to $\|\delta(t)\|_{A^0}$.
To this aim, let us denote $\|(a,b)\|_{A^k}:=\|a\|_{A^k}+\|b\|_{A^k}$ for $k=0,\ldots,4$. Since $\delta_\theta=\theta u+(1-\theta)v$ with $\theta\in[0,1]$, we have 
\begin{equation}
    \label{estim:deltath}
    \|\delta_\theta(t)\|_{A^k} \le \|\big(u(t),v(t)\big)\|_{A^k}\q(0\le k\le4)\qq \text{ and }\qq  \|\delta_\theta(t)\|_{A^0} \lesssim \|(u_0,v_0)\|_{A^0}\,,
\end{equation}
where we have used the boundedness property \eqref{eq:t0-stab} for the second inequality. 
Let us also define  
\begin{align*}
\omega_{0,\eta} &:= \eta \left(\,\sum_{j=1}^3 \|(u_0,v_0)\|_{A^0}^{4q_j/j}\right) + \|(u_0,v_0)\|_{A^0}^{q_4} \\ 
\mbox{and }\qquad 
g_\eta(t) &:= \|(u(t),v(t))\|_{A^0}^{q_0}\left(\,\sum_{k=1}^4 \|\big(u(t),v(t)\big)\|_{A^k}^{4/k}\right) + 
\sum_{j=1}^3 C_j(\eta) \left(\,\sum_{k=1}^4 \|\big(u(t),v(t)\big)\|_{A^k}^{4/k}\right).
\end{align*}
Summing up \eqref{estim:p0}--\eqref{estim:pj}, using the bounds \eqref{estim:deltath} and in view of the above definitions, we infer that  
\[
\sum_{j=0}^4 \|\nabla_{p_j}N[\delta_\theta](t)\cdot D^j\delta(t)\|_{A^0} \,\lesssim\, \omega_{0,\eta}\norm{\de(t)}_{A^4}+g_\eta(t)\norm{\de(t)}_{A^0}.
\]
Since the right-hand side is independent of $\theta$, we now deduce from \eqref{eq:diff} and \eqref{estim:nunv1} that, for some universal constant $C>0$,
\[
\frac{\dd}{\dd t}\norm{\de(t)}_{A^0}+\al\norm{\de(t)}_{A^4} \le
C\,\omega_{0,\eta}\norm{\de(t)}_{A^4}+C\,g_\eta(t)\norm{\de(t)}_{A^0}.
\]
Recall that, by assumption, $\norm{(u_0,v_0)}_{A^0}\,\leq\,2\veps_1$.
Hence, because $q_4>0$, 
we can ensure the inequality $C\omega_{0,\eta}\,\leq\,\alpha/2$ by taking $\eta>0$ and $\veps_1>0$ sufficiently small. We thus obtain
\begin{align} \label{est:pre-Gronw}
&\frac{\dd}{\dd t}\norm{\de(t)}_{A^0}\,+\,\frac{\al}{2}\,\norm{\de(t)}_{A^4}\,\le\, C\,g_\eta(t)\,\norm{\de(t)}_{A^0}\,.
\end{align}
Now, since $u$ and $v$ belong to $X_{0,4}$, we deduce from \eqref{eq:interpol_X0} that $g_\eta\in L^1(\R_+)$ for any $\eta>0$. An application of the Grönwall inequality thus yields the sought estimate.
\end{proof}

\section{Proof of existence} \label{s:proof-exist}

In this section, given any initial datum $u_0\in A^0_0(\T^d)$, we prove the existence of a solution $u\in X_{0,4}$ to the Cauchy problem \eqref{eq:parab}.
In particular, together with the uniqueness part shown in Section \ref{s:proof-unique}, this argument will conclude the proof of Theorem \ref{teo:wp}.

\medbreak

We use a fixed point argument in the mild formulation of equation \eqref{eq:parab}. Precisely, after passing in Fourier variables
and defining the semi-group $\big(e^{-\al t\D^2}\big)_{t\ge0}$ as a Fourier multiplier, namely
\[
\forall f\in\mc D'(\TT^d),\qq\forall\,k\in\ZZ^d,\qq \mf F\left[e^{-\al t\D^2}f\right](k)\, :=\, e^{-\al t|k|^4}\, \widehat{f}(k)\,,
\]
we see that equation \eqref{eq:parab} can be written as
\[u\,=\,\Psi[u]\,,\]
where we have defined the operator $\Psi[u]$ as
\begin{equation}
    \label{mild}
    \Psi[u](t)\, :=\, e^{-\al t\D^2}u_0\, +\, \int_0^t e^{-\al (t-\t)\D^2}\,N\big(\t,\cdot,u(\t),Du(\t),\ldots,D^4u(\t)\big) \;\dd\t\,.
\end{equation}

In the next subsection, we collect some useful preliminary lemmas. The actual proof of existence is carried out in Subsection \ref{sub:proofex}.

\subsection{Preliminary estimates} \label{ss:prelim}

We start with the following linear estimate.
\begin{lemma}\label{estim-sg}
    Let $a\in A^0_0(\TT^d)$. Then, we have $e^{-\al (\cdot)\D^2}a\in X_{0,4}$, together with the estimate 
    \begin{equation}
        \label{eq:estim-sg}
        \|e^{-\al (\cdot)\D^2}a\|_{X_{0,4}} \le C(\al)\, \|a\|_{A^0},
    \end{equation}
    where $C(\al):=\max\{1,1/\al\}$.
\end{lemma}

\begin{proof}
Observe that, by definition, one directly finds, for all $t\geq0$, that
\[
 \mf F\left[e^{-\al t\D^2}a\right](0)\,=\,\what a(0)\,.
\]
Thus, under the assumption that $a\in A^0_0(\TT^d)$, we infer that $\lan e^{-\alpha t\D^2}a\ran = 0$ for all $t\geq0$ as well.

Next, by definition again, we have 
\[
\|e^{-\al t\D^2}a\|_{A^0} = \sum_{k\in\ZZ^d\setminus\{0\}} e^{-\al t|k|^4} |\widehat{a}(k)| \le \sum_{k\in\ZZ^d\setminus\{0\}} |\widehat{a}(k)| = \|a\|_{A^0}\,,
\]
from which we deduce, on the one hand, that $\|e^{-\al (\cdot)\D^2}a\|_{L^\infty(\R_+;A^0)}\le \|a\|_{A^0}$ and, on the other hand,
by the dominated convergence theorem, that the map $t\,\mapsto\,e^{-\al t\D^2}a$ is continuous with values in $A^0_0(\TT^d)$.
Next, we write
\[
\|e^{-\al t\D^2}a\|_{A^4} = \sum_{k\in\ZZ^d\setminus\{0\}} |k|^4 e^{-\al t|k|^4} |\widehat{a}(k)|.
\]
Using this formula, we can compute
\[
\int_0^{+\infty} \|e^{-\al t\D^2}a\|_{A^4}\;\dd t = \sum_{k\in\ZZ^d\setminus\{0\}} |\widehat{a}(k)| \int_0^{+\infty} |k|^4 e^{-\al t|k|^4} \;\dd t = \frac1\al \sum_{k\in\ZZ^d\setminus\{0\}} |\widehat{a}(k)| = \frac1\al \,\|a\|_{A^0},
\]
so that $\|e^{-\al (\cdot)\D^2}a\|_{L^1(\R_+;A^4)}= (1/\al) \|a\|_{A^0}$. This completes the proof. 
\end{proof}

We now estimate the integral term in the mild formulation \eqref{mild}.

\begin{lemma}\label{estim-int}
Let $f\in L^1\big(\R_+;A^0_0(\TT^d)\big)$. Define the functional 
\[
\mc I[f](t) := \int_0^t e^{-\al (t-\t)\D^2}f(\t) \;\dd \t\,.
\]
Then, we have $\mc I[f]\in X_{0,4}$, together with the estimate 
\begin{equation}
    \label{eq:estim-int}
    \|\mc I[f]\|_{X_{0,4}} \le C(\al)\, \|f\|_{L^1(\R_+;A^0)}\,,
\end{equation}
where $C(\al):=\max\{1,1/\al\}$.
\end{lemma}

\begin{proof}
    By definition, for all $k\in\ZZ^d$ we have 
    \[
    \widehat{\mc I[f]} (t,k) = \int_0^t e^{-\al(t-\t)|k|^4}\widehat f(\t,k) \;\dd \t\,.
    \]
Then, since $\lan f(t)\ran=0$ for a.e. $t\geq0$, similarly to what happened in Lemma \ref{estim-sg} above we easily find that,
for all $t\geq0$, one has $\lan \mc I[f](t)\ran=0$ for all $t\geq0$. In addition, we have the bound
    \[
    \left|\widehat{\mc I[f]} (t,k)\right| \le \int_0^t |\widehat f(\t,k)| \;\dd \t\,,
    \]
which implies the estimate
    \[
    \big\|\mc I[f](t)\big\|_{A^0} = \sum_{k\in\ZZ^d\setminus\{0\}} \left|\widehat{\mc I[f]} (t,k)\right| \le 
\int_0^t \sum_{k\in\ZZ^d\setminus\{0\}}\left|\widehat f(\t,k)\right| \;\dd \t\,\le\, \int_0^{+\infty} \|f(\t)\|_{A^0} \; \dd \t\,.
    \]
Thus, we finally get $\left\|\mc I[f]\right\|_{L^\infty(\R_+;A^0)}\le \|f\|_{L^1(\R_+;A^0)}$, together with the property that
the map $t\,\mapsto\,\mc I[f](t)$ is continuous with values in $A^0_0$ (once again, by the Lebesgue dominated convergence theorem).

    Next, by definition we have 
    \[
    \left\|\mc I[f](t)\right\|_{A^4} = \sum_{k\in\ZZ^d\setminus\{0\}} |k|^4 \left|\int_0^t e^{-\al(t-\t)|k|^4}\widehat f(\t,k) \;\dd \t\right|\,.
    \]
From this relation, we can compute
    \[
    \begin{split}
        \int_0^{+\infty} \|\mc I[f](t)\|_{A^4} \;\dd t &\le 
        \sum_{k\in\ZZ^d\setminus\{0\}} \int_0^{+\infty} \int_0^t |k|^4 e^{-\al(t-\t)|k|^4} |\widehat f(\t,k)| \;\dd \t \;\dd t \\
        &= \sum_{k\in\ZZ^d\setminus\{0\}} \int_0^{+\infty} |\widehat f(\t,k)|\left(\int_\t^{+\infty} |k|^4 e^{-\al(t-\t)|k|^4} \dd t\right)  \dd \t \\
        &= \sum_{k\in\ZZ^d\setminus\{0\}} \int_0^{+\infty} |\widehat f(\t,k)|\left(\int_0^{+\infty} |k|^4 e^{-\al\tau'|k|^4} \;\dd\tau'\right)  \dd \t \\
        &= \frac1\al \int_0^{+\infty} \sum_{k\in\ZZ^d\setminus\{0\}} |\widehat f(\t,k)| \;\dd \t  \,= \,\frac1\al \int_0^{+\infty} \|f(\t)\|_{A^0} \;\dd \t\,,
    \end{split}
    \]
 which in turn implies the bound $\left\|\mc I[f]\right\|_{L^1(\R_+;A^4)}\le (1/\al) \|f\|_{L^1(\R_+;A^0)}$. This completes the proof of the lemma. 
\end{proof}

\subsection{Proof of existence}\label{sub:proofex}

As explained at the beginning of this section, given a fixed initial datum $u_0\in A^0_0(\TT^d)$, we reformulate equation \eqref{eq:parab} in
a mild form and perform a fixed point argument: we see a solution of \eqref{eq:parab} as a fixed point
$u=\Psi[u]$, where $\Psi[u]$ is defined by \eqref{mild}.

For simplicity of notation, similarly to Section \ref{s:proof-unique} we set
\[
N[u](t,x)\, :=\, N\big(t,x,u(t,x),Du(t,x),\ldots,D^4u(t,x)\big)\,.
\]
Then, we can rewrite $\Psi[u]$ as 
\begin{equation}
    \label{mild2}
    \Psi[u](t) \,= \,e^{-\al t\D^2}u_0\, +\, \mc I\big[N[u]\big](t)\,.
\end{equation}

For any $r>0$, let us denote the closed ball $\oline B_r := \Big\{u\in X_{0,4}\;\Big|\;\|u\|_{X_{0,4}}\le r\Big\}$.
We start by showing that $\Psi:\oline B_R\to \oline B_R$ for some $R>0$ suitably small.

Let us take some $u\in \oline B_R$, with $R\in\,]0,1[\,$. The linear term, namely the one depending on the initial datum only, is already estimated by Lemma
\ref{estim-sg}. Hence, we focus on the non-linear term: an application of Lemma \ref{estim-int} yields 
\[
\big\|\mc I\big[N[u]\big]\big\|_{X_{0,4}} \le C(\al)\, \|N[u]\|_{L^1(\R_+;A^0)}\,,
\]
so it remains to estimate $\|N[u]\|_{L^1(\R_+;A^0)}$.
From assumption \tbf{(A1)}, which we can apply by assuming $R$ sufficiently small, for a.e. $t\geq 0$ we can bound
\[
\|N[u](t)\|_{A^0} \lesssim \sum_{j=1}^4 \|u(t)\|_{A^0}^{\ell_j}\,\|u(t)\|_{A^j}^{4/j}\,.
\]
Making use of the interpolation inequality \eqref{interp-ex} and defining $\ell_j^\ast:=\ell_j+4/j-1$, we then gather that 
\[
\|N[u](t)\|_{A^0} \lesssim  \sum_{j=1}^4 \|u(t)\|_{A^0}^{\ell_j^\ast}\,\|u(t)\|_{A^4}\,.
\]
Notice that, because $\ell_{1,2,3}\in\NN$ and $\ell_4>0$, we have $\ell_j^\ast>0$ for all $j=0,\ldots,4$.
Let us now denote $\ell^*:=\min\big\{\ell_j^\ast\;\big|\;0\le j\le4\big\}>0$. Because $u\in \oline B_R$ with $R\in\,]0,1[\,$, we can estimate
\[
\|N[u]\|_{L^1(\R_+;A^0)} \le \sum_{j=1}^4 \|u\|_{L^\infty(A^0)}^{\ell_j^\ast}\,\|u\|_{L^1(A^4)} \le 4\,\|u\|_{X_{0,4}}^{\ell^*}\|u\|_{X_{0,4}}\,.
\]
From the previous considerations, we deduce that there exists a constant $K(\al)>0$, depending only on the parameter $\al>0$, such that 
\begin{equation}
    \label{ineg-psi}
    \|\Psi[u]\|_{X_{0,4}} \le K(\al)\,\left(\|u_0\|_{A^0} + \|u\|_{X_{0,4}}^{\ell^*}\|u\|_{X_{0,4}}\right)\,.
\end{equation}
Now, take $R>0$ sufficiently small that $K(\al)\,R^{\ell^*}\,\le\,1/2$ and suppose that $\|u_0\|_{A^0}\,\le\, R/2K(\al)$.
Since $u\in \oline B_R$, the previous estimate implies that $\|\Psi[u]\|_{X_{0,4}} \le R$. Thus, we have proven $\Psi:\oline B_R\to \oline B_R$, as claimed. 

Next, we show that $\Psi:\oline B_R\to \oline B_R$ is a contraction map. Let $u,v\in \oline B_R$: it follows from our definitions that
\[
\Psi[u]-\Psi[v]\, =\,\mc I\big[N[u]-N[v]\big]\,.
\]
This implies, from Lemma \ref{estim-int}, that 
\begin{equation}
    \label{estim:pupv}
    \|\Psi[u]-\Psi[v]\|_{X_{0,4}} \le C(\al)\,\|N[u]-N[v]\|_{L^1(\R_+;A^0)}\,.
\end{equation}
Similarly to Section \ref{s:proof-unique}, let us now denote $\delta:=u-v$ and, for $\theta\in[0,1]$, $\delta_\theta:=v+\theta(u-v)$.
We can then write the difference 
\[
N[u]-N[v] = \sum_{j=0}^4 \,\int_0^1 \nabla_{p_j}N[\delta_\theta]\cdot D^j\delta \;\dd\theta\,,
\]
which can be bounded as 
\begin{equation}
    \label{estim:nunv}
    \|N[u]-N[v]\|_{L^1(\R_+;A^0)} \le \sum_{j=0}^4 \,\int_0^1 \|\nabla_{p_j}N[\delta_\theta]\cdot D^j\delta\|_{L^1(\R_+;A^0)} \;\dd\theta\,.
\end{equation}
We now estimate the $L^1(\R_+;A^0)$ norm appearing inside the integral above.
By taking $R$ smaller if needed, we use assumption \tbf{(A2)} and the Banach algebra property of $A^0$ to get, for a.e. $t\geq 0$, the estimate
\[
\|\nabla_{p_j}N[\delta_\theta](t)\cdot D^j\delta(t)\|_{A^0}
\lesssim \|\delta_\theta(t)\|_{A^0}^{q_j}\left(\,\sum_{k=1}^4 \|\delta_\theta(t)\|_{A^k}^{(4-j)/k}\right)\|\delta(t)\|_{A^j}\,.
\]
Applying the interpolation inequality \eqref{interp-ex} twice, we thus gather
\[
\begin{split}
    &\|\nabla_{p_j}N[\delta_\theta](t)\cdot D^j\delta(t)\|_{A^0} \\
    &\lesssim \|\delta_\theta(t)\|_{A^0}^{q_j}\left(\,\sum_{k=1}^4 \|\delta_\theta(t)\|_{A^0}^{(4-j)/k-(4-j)/4} \|\delta_\theta(t)\|_{A^4}^{(4-j)/4}\right)
    \|\delta(t)\|_{A^0}^{1-j/4}\, \|\delta(t)\|_{A^4}^{j/4}  \,.
\end{split}
\]
Using the H\"older inequality, we arrive at the bound
\[
\begin{split}
    &\|\nabla_{p_j}N[\delta_\theta]\cdot D^j\delta\|_{L^1(\R_+;A^0)} \\
    &\qquad \lesssim \|\delta_\theta\|_{L^\infty(\R_+;A^0)}^{q_j}
    \left(\,\sum_{k=1}^4\|\delta_\theta\|_{L^\infty(\R_+;A^0)}^{(4-j)/k-(4-j)/4}\,\|\delta\|_{L^\infty(\R_+;A^0)}^{1-j/4}\right)
    \int_0^{+\infty}  \|\delta_\theta(t)\|_{A^4}^{(4-j)/4}\, \|\delta(t)\|_{A^4}^{j/4} \;\dd t \\
    &\qquad \lesssim \|\delta_\theta\|_{L^\infty(\R_+;A^0)}^{q_j}
    \left(\,\sum_{k=1}^4\|\delta_\theta\|_{L^\infty(\R_+;A^0)}^{(4-j)/k-(4-j)/4}\|\delta\|_{L^\infty(\R_+;A^0)}^{1-j/4}\right)
    \|\delta_\theta\|_{L^1(\R_+;A^4)}^{(4-j)/4} \|\delta\|_{L^1(\R_+;A^4)}^{j/4} \\
    &\qquad \lesssim \left(\,\sum_{k=1}^4\|\delta_\theta\|_{X_{0,4}}^{q_{j,k}^\ast}\right)\|\delta\|_{X_{0,4}}\,,
\end{split}
\]
where we have defined $q_{j,k}^\ast:=q_j+(4-j)/k$. Notice that, because $q_{1,2,3}\in\NN$ and $q_4>0$, we have $q_{j,k}^\ast>0$ for all $j=0,\ldots,4$ and $k=1,\ldots,4$.
Since we can rewrite $\delta_\theta=\theta u+(1-\theta)v$ and since $u,v\in \oline B_R$, we also have $\|\delta_\theta\|_{X_{0,4}}\le \|u\|_{X_{0,4}}+\|v\|_{X_{0,4}}\le2R$.

Suppose now that $R<1/2$, so that $\|\delta_\theta\|_{X_{0,4}}\leq 2R<1$. Then, denote $q^\ast:=\min\big\{q_{j,k}^\ast\;\big|\;0\le j,k\le4\big\}$
and observe that $q^\ast>0$. From the previous estimate, we then infer, for some universal constant $C>0$, the bound
\[
\big\|\nabla_{p_j} N[\delta_\theta]\cdot D^j\delta\big\|_{L^1(\R_+;A^0)} \le C\,\|\delta_\theta\|_{X_{0,4}}^{q^\ast}\,
\|\delta\|_{X_{0,4}} \le C\,(2R)^{q^\ast}\,\|\delta\|_{X_{0,4}}\,.
\]
Plugging this into \eqref{estim:nunv} and recalling that $\delta=u-v$, in turn we obtain
\[
\|N[u]-N[v]\|_{L^1(\R_+;A^0)} \,\le\, 5C\,(2R)^{q^\ast}\,\|u-v\|_{X_{0,4}}\,.
\]
From \eqref{estim:pupv}, we finally deduce that
\[
\|\Psi[u]-\Psi[v]\|_{X_{0,4}}\, \le\, C(\al)\,5C\,\cdot(2R)^{q^\ast}\,\|u-v\|_{X_{0,4}}\,.
\]
Now, take $R\in\,]0,1/2[\,$ sufficiently small that $C(\al)\,5C\,(2R)^{q^\ast}<1$.
Then, $\Psi:\oline B_R\to \oline B_R$ is a contraction map in the Banach space $\oline B_R$. Therefore, $\Psi$ has a unique fixed point $u\in \oline B_R$,
which solves the equation $u=\Psi(u)$, or equivalently equation \eqref{eq:parab}.

The previous argument shows the existence of a solution to equation \eqref{eq:parab}.
The uniqueness of this solution in the whole space $X_{0,4}$ has been proved in Section \ref{s:proof-unique}.
Moreover, estimate \eqref{est:solution} is a direct consequence of inequality \eqref{ineg-psi} (recall that $R$ has been further fixed therein so that
$K(\al)R^{\ell^*}\le1/2$) and of the fact that $u=\Psi(u)$.
The proof of Theorem \ref{teo:wp} is thus completed.

\section{Generalisations} \label{s:general}

In this section, we propose some possible generalisations of the theory developed for problem \eqref{eq:parab}.
We first extend the results of Theorems \ref{teo:wp} and \ref{t:stab} to the case of more regular initial data, \tsl{i.e.} of data such
that $u_0\in A^s_0(\T^d)$ for some $s>0$. Then, we show gain of analyticity of the solution for any $t>0$.
Finally, we study the case of higher-order parabolic operators.
In all these parts, we will need do suitable adapt assumptions {\bf(A1)}, {\bf (A2)} and {\bf (A3)} to the new framework.

How to consider combinations of these three aspects (for instance, propagation of higher regularity, or gain of analyticity, for higher-order operators)
will easily follow from our discussion, and is left to the interested reader.

\subsection{Propagation of higher regularities} \label{ss:higher-reg}

In this subsection, we consider again the problem in \eqref{eq:parab}, but now assume that the initial datum enjoys higher regularity, {\sl i.e.}
we assume $u_0\in A^s_0(\T^d)$, for some $s>0$.
In this case, we look for solutions belonging to
\begin{equation*}\label{eq:Xs4}
X_{s,4+s}\,:=\,
\mc C_b\big(\!\R_+;A^s_0(\T^d)\big)\,\cap\,L^1\big(\!\R_+;A^{4+s}(\T^d)\big)\,.
\end{equation*}
By making use, for $s_1<s_2$ and $s_1\leq \s\leq s_2$, of the interpolation inequality
\begin{equation}\label{interpol2}
\|u\|_{A^\s}\,\le\,\|u\|_{A^{s_1}}^{1-\theta}\,\|u\|_{A^{s_2}}^{\theta}\,,\qquad\qquad \mbox{ for }\qquad \theta=\frac{\s-s_1}{s_2-s_1}\,,
\end{equation}
it is easy to see that, if $u\in X_{s,4+s}$, then one also has
\begin{equation}\label{eq:interpol_Xs}
 u\,\in\, L^{4}\big(\!\R_+;A^{s+1}_0(\T^d)\big)\,\cap\,L^{2}\big(\!\R_+;A^{s+2}(\T^d)\big)\,\cap\,L^{4/3}\big(\!\R_+;A^{s+3}(\T^d)\big)\,.
\end{equation}

In the following statements, the assumptions {\rm \tbf{(A1)}} and {\rm \tbf{(A2)}} are the same as those stated in Subsection \ref{ss:assumpt}, with the exception that the $A^0$ norms are now replaced by $A^s$ norms, for some $s>0$. For the sake of conciseness, we do not restate them with different names.
Since these results can be obtained by rather direct adaptations of the arguments developed in Sections \ref{s:proof-unique} and \ref{s:proof-exist}, we omit their proofs.

\begin{theorem}[Well-posedness]\label{teo:wps}
 Let $s>0$ fixed.
Consider the Cauchy problem \eqref{eq:parab}, under assumptions {\rm \tbf{(A1)}}, {\rm \tbf{(A2)}} and {\rm \tbf{(A3)}}
on the non-linearity $N$ (with the agreement that the $A^0$ norms in {\rm \tbf{(A1)}} and {\rm \tbf{(A2)}} are replaced by the $A^s$ norms).

Then, there exists $\veps_0>0$ such that, for any $u_0\in A^s_0(\T^d)$ satisfying $\left\|u_0\right\|_{A^s}\,\leq\,\veps_0$,
the unique solution $u\in X_{0,4}$ to the Cauchy problem \eqref{eq:parab}, given by Theorem \ref{teo:wp}, satisfies in addition the property
$u\in X_{s,4+s}$.
Moreover, there exists a constant $C\,=\,C(\alpha,\veps_0)\,>\,0$ 
such that
\begin{equation} \label{est:solutions}
 \|u\|_{X_{s,4+s}}\,\leq\,C\,\left\|u_0\right\|_{A^s}\,.
\end{equation}
\end{theorem}

\begin{theorem}[Stability] \label{t:stabs}
Let $s>0$ fixed.
Consider the Cauchy problem \eqref{eq:parab}, under assumptions {\rm \tbf{(A1)}}, {\rm \tbf{(A2)}} and {\rm \tbf{(A3)}}
on the non-linearity $N$ (with the agreement that the $A^0$ norms in {\rm \tbf{(A1)}} and {\rm \tbf{(A2)}} are replaced by the $A^s$ norms).
Let $u_0$ and $v_0$ be two initial data in $A^s_0(\T^d)$, and let respectively $u$ and $v$ be two
corresponding solutions belonging to the space $X_{s,4+s}$.
Assume that
\[
\norm{u}_{L^\infty(\R_+;A^s)}\,\lesssim\,\norm{u_0}_{A^s}\qq \text{ and } \qq \norm{v}_{L^\infty(\R_+;A^s)}\,\lesssim\,\norm{v_0}_{A^s}\,.
\]

Then, there exists $\veps_1>0$ such that, if $\norm{u_0}_{A^s}\,\leq\,\veps_1$ and $\norm{v_0}_{A^s}\,\leq\,\veps_1$, one has
\[
\|u-v\|_{X_{s,4+s}}\,\lesssim\,\norm{u_0-v_0}_{A^s}\,, 
\]
where the implicit constant may depend on $\alpha$, $\veps_1$ and on the norms of $u$ and $v$ in $X_{s,4+s}$.
\end{theorem}

\subsection{Gain of analyticity} \label{ss:analyticity}

In the present subsection, we still consider problem \eqref{eq:parab} for a low regularity initial datum $u_0\in A^0_0(\T^d)$.
We show gain of analyticity of the solution globally in time. Similar ideas have been used before
in analogous problems (see for instance \cite{granerosema,LiuStrain} and the references therein).

In order to set up the problem and rigorously state the result, some preparation is in order.
Given real numbers $s\geq0$ and $\nu\ge0$, let us define the space
\begin{align*}
A^{s,\nu}(\TT^d)\, &:=\,
\bigg\{f\,\in\, L^1(\TT^d)\quad \Big|\quad  \|f\|_{A^{s,\nu}}\,:=\,\sum_{k\in\ZZ^d\setminus\{0\}} e^{\nu|k|} |k|^s\;|\widehat{f}(k)|\,<\,+\,\infty\;\bigg\}\,.
\end{align*}
Functions belonging to the space $A^{0,\nu}$ with $\nu>0$ are thus analytic, with radius of analyticity greater or equal to $\nu$.
We also introduce the subspace of $A^{s,\nu}(\TT^d)$ composed of mean-free functions, namely 
\[
A^{s,\nu}_0(\TT^d) := \Big\{f\in A^{s,\nu}(\TT^d)\;\;\big|\;\;\lan f\ran=0\Big\}.
\]
 
Now, we consider the case in which $\nu=\nu(t)$ is a time-dependent function from $\R_+$ into itself.
More precisely, here we fix the choice $\nu(t)\,=\,\beta t$, for some constant $\beta>0$, which is the only relevant case for us
(see Theorem \ref{teo:analyticity}); however, in the discussion below, the important point is that the function $\nu(t)$ is increasing with respect to $t$.
As a matter of fact, in this case we notice that, for fixed $s\geq0$,
the spaces $A^{s,\nu(t)}$ are decreasing (with respect to the inclusion relation) in time. 
Inspired by the study of
hyperbolic operators with low regularity coefficients \cite{C-M, C-DS-F-M} (with the crucial difference that the spaces are increasing, there),
we introduce the following definitions.
First of all, in a natural way, for any $p\in[1,+\infty]$ the notation $v\in L^p\big(\R_+;A^{s,\nu(\cdot)}(\TT^d)\big)$ means that 
$v(t)\in A^{s,\nu(t)}(\TT^d)$ for a.e. $t\in\R_+$, together with the property that the map
\[
t\,\mapsto\,\left\|u(t)\right\|_{A^{s,\nu(t)}}\qquad\qquad \mbox{ belongs to }\qquad L^p(\R_+)\,.
\]
The subset $L^p_{\rm loc}\big(\R_+;A^{s,\nu(\cdot)}(\TT^d)\big)$ is defined in a similar way, but restricting the attention to compact time intervals $[0,T]$.
Next, with the above preparation in mind, we can define $\mf C\big(\R_+;A^{s,\nu(\cdot)}(\TT^d)\big)$ as the space of functions
$v\in L^\infty_{\rm loc}\big(\R_+;A^{s,\nu(\cdot)}(\TT^d)\big)$ such that, for any $t_0\geq0$ fixed, one has
$v\in \mc C\big([t_0,+\infty[\,;A^{s,\nu(t_0)}(\TT^d)\big)$. Observe that, now, the target space $A^{s,\nu(t_0)}(\TT^d)$ is fixed with respect to the time variable,
so the definition of time continuity with values into it is not ambiguous.
We also set $\mf C_b\big(\R_+;A^{s,\nu(\cdot)}(\TT^d)\big)$ the subspace of the previous
space, composed by functions $v$ which are $L^\infty\big(\R_+;A^{s,\nu(\cdot)}(\TT^d)\big)$.
It goes without saying that we can replace the target space $A^{s,\nu(\cdot)}(\TT^d)$ by $A^{s,\nu(\cdot)}_0(\TT^d)$ in all the definitions above.
Then, we can finally define the space
\begin{equation*}\label{eq:Xs4anal}
\mc X^{\nu(\cdot)}_{0,4}\,:=\,\mf C_b\big(\R_+;A^{0,\nu(\cdot)}_0(\TT^d)\big)\,\cap\, L^1\big(\R_+;A^{4,\nu(\cdot)}(\TT^d)\big)\,,
\end{equation*}
endowed with the natural norm
\[
 \left\|v\right\|_{\mc X_{0,4}^{\nu(\cdot)}}\,:=\,\sup_{t\in\R_+}\left\|v(t)\right\|_{A^{0,\nu(t)}}\,+\,\int^{+\infty}_0\left\|v(t)\right\|_{A^{4,\nu(t)}}\,\dt\,.
\]
To conclude, we remark that $\mc X_{0,4}^{\nu(\cdot)}\,\subset\, X_{0,4}$, where $X_{0,4}$ is the space defined in \eqref{def:X_{0,4}}.

\medbreak
After these preliminaries, we  now state our result.
In the following statement, the assumptions {\rm \tbf{(A1)}} and {\rm \tbf{(A2)}} are the same as those stated in Subsection \ref{ss:assumpt}, with the exception that the $A^0$ norms are now replaced by $A^{0,\nu(t)}$ norms, for some suitable function $\nu(t)$ defined below. For the sake of conciseness,
we do not restate them here.

\begin{theorem}[Gain of analyticity] \label{teo:analyticity}
Given $\alpha>0$, let $0<\beta<\alpha$ and define, for $t\geq0$, the function $\nu(t)\,:=\,\beta\,t$.
Consider the Cauchy problem \eqref{eq:parab} under assumptions {\rm \tbf{(A1)}}, {\rm \tbf{(A2)}} and {\rm \tbf{(A3)}}
on the non-linearity $N$ (with the agreement that the $A^0$ norms in {\rm \tbf{(A1)}} and {\rm \tbf{(A2)}} are replaced by the $A^{0,\nu(t)}$ norms).

Then, there exists $\veps_0>0$ such that, for any $u_0\in A^0_0(\T^d)$ satisfying $\left\|u_0\right\|_{A^0}\,\leq\,\veps_0$,
the unique solution $u\in X_{0,4}$ to the Cauchy problem \eqref{eq:parab}, given by Theorem \ref{teo:wp}, satisfies in addition the property
$u\in \mc X_{0,4}^{\nu(\cdot)}$.
In particular, for any fixed time $t>0$, the solution $u(t)$ at that time is analytic in the space variable.
Finally, there exists a constant $C\,=\,C(\alpha,\beta,\veps_0)\,>\,0$
such that
\begin{equation}\label{ineg:gain}
 \|u\|_{\mc X_{0,4}^{\nu(\cdot)}}\,\leq\,C\,\left\|u_0\right\|_{A^0}\,.
\end{equation}
\end{theorem}

\begin{proof}
    We show that the map $\Psi$ defined by \eqref{mild} is a contraction of a suitably small ball of $\mc X_{0,4}^{\nu(\cdot)}$.
    For this, we need to adapt Lemmas \ref{estim-sg} and \ref{estim-int} to the present framework.
We point out that the proof of the time continuity of both the initial datum and non-linear terms follows the same lines as in Subsection \ref{ss:prelim} above,
therefore we omit to treat it here.

To begin with, we observe that, since $0<\beta<\alpha$, one has
    \[
    \forall\,k\in\Z^d\setminus\{0\}\,, \qq \alpha|k|^3-\beta>0\,.
    \]
    
    The linear part of \eqref{mild} can be estimated, for a.e. $t\geq0$, as 
    \begin{align*}
    \left\|e^{-\al t\D^2}a\right\|_{A^{0,\nu(t)}} &= \sum_{k\in\ZZ^d\setminus\{0\}}e^{-\al t|k|^4}e^{\nu(t)|k|}\, |\widehat{a}(k)| \\
    &= 
    \sum_{k\in\ZZ^d\setminus\{0\}} e^{-t|k|(\al|k|^3-\beta)}\, |\widehat{a}(k)|  \le \sum_{k\in\ZZ^d\setminus\{0\}} |\widehat{a}(k)| = \|a\|_{A^0}\,.
    \end{align*}
Similarly, we can compute
    \[
    \begin{split}
        \int_0^{+\infty} \|e^{-\al t\D^2}a\|_{A^{4,\nu(t)}}\;\dd t 
        &= \sum_{k\in\ZZ^d\setminus\{0\}} |k|^4|\widehat{a}(k)| \int_0^{+\infty} e^{-t|k|(\al|k|^3-\beta)} \;\dd t \\
&= \sum_{k\in\ZZ^d\setminus\{0\}} \frac{|k|^3}{\al|k|^3-\beta}|\widehat{a}(k)| \le \sum_{k\in\ZZ^d\setminus\{0\}} \frac{|k|^3}{\al|k|^3-\beta|k|^3}|\widehat{a}(k)| 
= \frac{1}{\al-\beta} \,\|a\|_{A^0}\,.
    \end{split}
    \]
Gathering these two inequalities together, we deduce that
$\|e^{-\al (\cdot)\D^2}a\|_{\mc X_{0,4}^{\nu(\cdot)}} \le C(\al,\beta)\,\|a\|_{A^0}$, with $C(\al,\beta):=\max\{1,1/(\al-\beta)\}$.

    We now estimate the integral term appearing in \eqref{mild}. We recall that the functional defined in Lemma \ref{estim-int} associated to this part is given by 
    \[
    \mc I[f](t)\, :=\, \int_0^t e^{-\al (t-\t)\D^2}f(\t) \;\dd \t\,.
    \]
    First of all, we have 
    \[
    \begin{split}
        \left\|\mc I[f](t)\right\|_{A^{0,\nu(t)}} 
        &= \sum_{k\in\ZZ^d\setminus\{0\}} e^{\nu(t)|k|}\left|\widehat{\mc I[f]} (t,k)\right| \\
        &\le \int_0^t \sum_{k\in\ZZ^d\setminus\{0\}} e^{-(\al|k|^3-\beta)(t-\t)|k|}e^{\beta \t|k|}\left|\widehat f(\t,k)\right| \;\dd \t
        \le \int_0^{+\infty} \|f(\t)\|_{A^{0,\nu(\t)}} \; \dd \t\,.
    \end{split}
    \]
Moreover, we can also estimate
    \[
    \begin{split}
        \int_0^{+\infty} \left\|\mc I[f](t)\right\|_{A^{4,\nu(t)}} \;\dd t &\le 
        \sum_{k\in\ZZ^d\setminus\{0\}} \int_0^{+\infty} \int_0^t |k|^4 e^{-(\al|k|^3-\beta)(t-\t)|k|}e^{\beta \t|k|} \left|\widehat f(\t,k)\right| \;\dd \t \;\dd t \\
&= \sum_{k\in\ZZ^d\setminus\{0\}} \int_0^{+\infty} |k|^4e^{\beta \t|k|}\left|\widehat f(\t,k)\right|\left(\int_\t^{+\infty}  e^{-(\al|k|^3-\beta)(t-\t)|k|} \dd t\right)  \dd \t \\
&= \sum_{k\in\ZZ^d\setminus\{0\}} \int_0^{+\infty} |k|^4e^{\beta \t|k|}\left|\widehat f(\t,k)\right|\left(\int_0^{+\infty}  e^{-(\al|k|^3-\beta)t|k|} \;\dd t\right)  \dd \t \\
&= \int_0^{+\infty} \sum_{k\in\ZZ^d\setminus\{0\}} \frac{|k|^3}{\al|k|^3-\beta}\,e^{\beta \t|k|}\left|\widehat f(\t,k)\right| \;\dd \t \\
        &\le \frac{1}{\al-\beta} \int_0^{+\infty} \|f(\t)\|_{A^{0,\nu(\t)}} \;\dd \t\,.
    \end{split}
    \]
Combining these estimates, we infer that $\left\|\mc I[f]\right\|_{\mc X_{0,4}^{\nu(\cdot)}} \le C(\al,\beta)\,\|f\|_{L^1(\R_+;A^{0,\nu(\cdot)})}$.

    Let us now notice that the interpolation inequality \eqref{interp-ex} also holds for the spaces $A^{j,\nu}$, as can be seen by writing 
    \[
\|f\|_{A^{j,\nu}}\,=\,\sum_{k\in\ZZ^d\setminus\{0\}} \left(e^{\nu|k|}\left|\widehat{f}(k)\right|\right)^{1-j/4} \left(e^{\nu|k|}|k|^4\left|\widehat{f}(k)\right|\right)^{j/4}
    \]
    and using the H\"older inequality.

    Therefore, with the natural adaptations of assumptions {\rm \tbf{(A1)}} and {\rm \tbf{(A2)}} at hand, we can follow the same kind of computations as in Subsection \ref{sub:proofex} to gather that the map $\Psi$ is a contraction of a suitably small ball of $\mc X_{0,4}^{\nu(\cdot)}$; hence, it possesses a fixed point
    $u\in \mc X_{0,4}^{\nu(\cdot)}$, which solves equation \eqref{eq:parab} with initial datum $u_0\in A^0_0(\TT^d)$ and satisfies inequality \eqref{ineg:gain}. Since
    $\mc X_{0,4}^{\nu(\cdot)}\subset X_{0,4}$, owing to \eqref{ineg:gain}, we can apply the stability result of Theorem \ref{t:stab}: this implies that the just
    found solution must coincide with the solution given by Theorem \ref{teo:wp}. Besides, this also proves uniqueness.
\end{proof}

\subsection{The case of higher-order operators}\label{sec:2m}

Here we show how the techniques developed in the previous sections can be adapted to handle the case of higher-order operators.

We  generalise the fourth-order problem \eqref{eq:parab} as
\begin{equation}\label{eq:parab-m}
\begin{cases}
\begin{array}{ll}
\ds\pat u  \, +\,\al\,(-\D)^mu\,
=\,N(t,x,u,Du,\ldots,D^{2m} u) &\qquad  \mbox{ in }\q  \R_+\!\times\,\TT^d\,, \\
\ds u\bigl|_{t=0}\,=\,u_0 &\qquad \mbox{ in }\q\, \TT^d\,,
\end{array}
\end{cases}
\end{equation}
where $\al >0$ and $m\in\NN\setminus\{0\}$.

Of course, we also need to replace assumptions \tbf{(A1)}, \tbf{(A2)} and \tbf{(A3)} on the non-linearity $N$, as specified below.
In the following, we are going to use the shortcut $\mf p_n := (p_1,\ldots ,p_n)$; in particular, after defining $d_j:=d^j$ for all $1\le j\le n$ and
$\mf D_n:=\sum_{1\le j\le n} d_j$, we have $\mf p_n \in \R^{\mf D_n}$.
As done in Section \ref{s:results}, we also fix some $\kappa\in]0,+\infty]$.

\medbreak
\tbf{(A1.$m$)}
For any fixed $(p_0,\mf p_{2m})\,\in\;]-\k,\k[\,\times\R^{\mf D_{2m}}$, 
the map $(t,x)\,\mapsto\,N(t,x,p_0,\mf p_{2m})$ belongs to $L^\infty\big(\!\R_+;A^0(\T^d)\big)$.
In addition, there exist exponents $\ell_1,\ldots,\ell_{2m-1}\in\NN$ and $\ell_{2m}>0$ such that,
for a.e. $t\in\R_+$ and for all $v\in A^{2m}(\T^d)$ satisfying $\|v\|_{A^0}<\k$, one has
\[
\left\|N(t,\cdot,v,Dv,\ldots,D^{2m}v)\right\|_{A^0}\,
\lesssim\,
\sum_{j=1}^{2m}\left\|v\right\|_{A^0}^{\ell_j}\,\left\|D^jv\right\|_{A^0}^{2m/j}\,.
\]

\medbreak
\tbf{(A2.$m$)} For a.e. $(t,x)\in\R_+\times\T^d$, the function $N$ is \emph{locally Lipschitz continuous} with respect to the $(p_0,\mf p_{2m})$-variables
over $\,]-\k,\k[\,\times\R^{\mf D_{2m}}$. In addition, there exist
exponents  $q_0,\ldots,q_{2k-1}\in\NN$ and $q_{2m}>0$ such that,
for a.e. $t\in\R_+$ and for all $v\in A^{2m}(\T^d)$ satisfying $\|v\|_{A^0}<\k$, one has
\[
\left\|\nabla_{p_j}N(t,\cdot,v, Dv,\ldots,D^{2m}v)\right\|_{A^0}\,\lesssim\,\left\|v\right\|_{A^0}^{q_j}\,
\sum_{i=1}^{2m}\left\|D^iv\right\|_{A^0}^{(2m-j)/i} \quad \mbox{ for }\q 0\le j\le 2m\,.
\]

\medbreak
\tbf{(A3.$m$)} There exists a function $H\,=\, H(t,x,p_0,p_1,\ldots,p_{2m-1})$ such that, for a.e. $(t,x)\in\R_+\!\times\,\T^d$ and any
$v\in W^{2m,\infty}(\T^d)$, one has
\[
N(t,x,v,Dv,\ldots,D^{2m}v)\,=\,\div H(t,x,v,Dv,\ldots,D^{2m-1}v)\,,
\]
where the divergence operator acts only with respect to the space variable $x\in\T^d$.

\medbreak

Given an initial datum $u_0\in A^0_0(\TT^d)$, we are interested in obtaining well-posedness results for the Cauchy problem \eqref{eq:parab-m} in the space
\[
X_{0,2m}\,:=\,\mc C_b\big(\!\R_+;A^0_0(\T^d)\big)\,\cap\,L^1\big(\!\R_+;A^{2m}(\T^d)\big)\,.
\]
In this setting, Theorems \ref{teo:wp} and \ref{t:stab} can be generalised in the following way. We omit the details of the proofs, as they are based on rather
straightforward adaptations of the arguments developed in Sections \ref{s:proof-unique} and \ref{s:proof-exist}.

\begin{theorem}[Well-posedness]\label{teo:wp-m}
Let $m\in\NN\setminus\{0\}$. Consider the Cauchy problem \eqref{eq:parab-m}, under assumptions {\rm \tbf{(A1.$m$)}}, {\rm \tbf{(A2.$m$)}} and
{\rm \tbf{(A3.$m$)}} on the non-linearity $N$.

Then, there exists $\veps_0>0$ such that, for any $u_0\in A^0_0(\T^d)$ satisfying $\left\|u_0\right\|_{A^0}\,\leq\,\veps_0$,
there exists a unique solution $u\in X_{0,2m}$ to \eqref{eq:parab-m} related to this initial datum.
In addition, there exists a constant $C\,=\,C(\alpha,\veps_0)\,>\,0$ 
such that
\begin{equation*} \label{est:solutionm}
 \|u\|_{X_{0,2m}}\,\leq\,C\,\left\|u_0\right\|_{A^0}\,.
\end{equation*}
\end{theorem}

\begin{theorem}[Stability]\label{teo:stabm}
Let $m\in\NN\setminus\{0\}$. Consider the Cauchy problem \eqref{eq:parab-m}, under assumptions {\rm \tbf{(A1.$m$)}}, {\rm \tbf{(A2.$m$)}} and {\rm \tbf{(A3.$m$)}}
on the non-linearity $N$. Let $u_0$ and $v_0$ be two initial data in $A^0_0(\T^d)$, and let respectively $u$ and $v$ be two
corresponding solutions belonging to the space $X_{0,2m}$.
Suppose also that
\[
\norm{u}_{L^\infty(\R_+;A^0)}\,\lesssim\,\norm{u_0}_{A^0}\qq \text{ and }\qq \norm{v}_{L^\infty(\R_+;A^0)}\,\lesssim\,\norm{v_0}_{A^0}\,.
\]

Then, there exists $\veps_1>0$ such that, if $\norm{u_0}_{A^0}\,\leq\,\veps_1$ and $\norm{v_0}_{A^0}\,\leq\,\veps_1$, one has
\[
\|u-v\|_{X_{0,2m}}\,\lesssim\,\norm{u_0-v_0}_{A^0}\,, 
\]
where the implicit constant may depend on $\alpha$, $\eps_1$ and on the norms of $u$ and $v$ in $X_{0,2m}$.
\end{theorem}

\section{Applications} \label{s:applic}

In this last part of our work, we show how certain models appearing in the literature (see \tsl{e.g.} \citep {granero2018global,magliocca2024fourth,granero2024global})
enter the framework developed within the present work. In particular, Theorems \ref{teo:wp} and \ref{t:stab} apply to those models, thus filling
the existing gap between the functional settings where existence and uniqueness hold.

\subsection{The models from \citep{granero2018global}}\label{s:crystal}
The authors of \citep{granero2018global} deal  with several models for describing the evolution of crystal surfaces, which were proposed in \citep{krug1995adatom} and \citep{shehadeh2011evolution}. The main ones read as follows:
\begin{equation}\label{eq:18-1}
\begin{cases}
\begin{array}{ll}
\ds\pat v\,=\,\D e^{-\D v} & \qquad \mbox{ in }\q  \R_+\!\times\,\TT^d\,, \\
\ds v\bigl|_{t=0}\,=\,v_0 &\qquad \mbox{ in }\q\, \TT^d\,,
\end{array}
\end{cases}
\qquad
\text{with}\q
\int_{\TT^d} v_0(x)\,\dx=0,
\end{equation}
and
\begin{equation}\label{eq:18-2}
\begin{cases}
\begin{array}{ll}
\ds\pat v\,=\,-v^2\,\D^2(v^3) &\qquad \mbox{ in }\q  \R_+\!\times\,\TT^d\,,  \\
\ds v\bigl|_{t=0}\,=\,v_0 &\qquad \mbox{ in }\q\, \TT^d\,,
\end{array}
\end{cases}
\qquad
\text{with}\q
v_0> 0.
\end{equation}
In each case, a suitable change of variable $u=f(v)$ is needed. The existence of solutions verifying
\[
u\in L^{\infty}\big(\R_+;L^\infty(\TT^d)\big)\cap \m{M}\big(\R_+;W^{4,\infty}(\TT^d)\big),
\]
where $\mc M\big(\!\R_+;W^{4,\infty}(\TT^d)\big)$ denotes the space
of Radon measures over $\R_+$ taking values in $W^{4,\infty}(\TT^d)$,
and $\norm{u(t)}_{L^\infty}\lesssim\norm{u_0}_{A^0}$ assuming $\norm{u_0}_{A^0}<1$ is proved, respectively, in \citep[Theorems 2.2 and 2.6]{granero2018global}.

\subsubsection{Verifying {\rm\tbf{(A1)}}, {\rm\tbf{(A2)}} and {\rm\tbf{(A3)}} for \eqref{eq:18-1}}

Using the change of variable $u=\D v$, we get the following equivalent formulation of \eqref{eq:18-1}:
\begin{equation}\label{eq:18-1-bis}
    \begin{cases}
\begin{array}{ll}
\ds\pat u\,=\,\D^2e^{-u} &\qquad \mbox{ in }\q  \R_+\!\times\,\TT^d\,, \\
\ds u\bigl|_{t=0}\,=\,\D v_0 &\qquad \mbox{ in }\q\, \TT^d\,.
\end{array}
\end{cases}
\end{equation}
We can rewrite the right-hand side as    
\begin{equation}\label{verif-n1}
    \D^2e^{-u}\,=\,\D^2\pare{\sum_{r\ge0}(-1)^r\frac{u^r}{r!}}\,=\,-\,\D^2u\,+\,\D^2\pare{\sum_{r\ge2}(-1)^r\frac{u^r}{r!}}\,.
\end{equation}
Let us now define the function
\begin{align*}
 N_1(p_0,\mf p) := \sum_{s=1}^{4}\sum_{r\ge2}(-1)^r\tn_{s}^r(p_0,\mf p) ,
\end{align*}
where
\begin{align*}
\tn_{1}^r(p_0,\mf p)&:=\frac{1}{(r-1)!}p_0^{r-1}p_4,
&
\tn_{2}^r (p_0,\mf p)&:=\frac{1}{(r-2)!}p_0^{r-2}\left[ p_2^2+ p_3p_1\right],
\\
\tn_{3}^r (p_0,\mf p)&:=\frac{1}{(r-3)!}p_0^{r-3}p_2p_1^2 ,
&
\tn_{4}^r (p_0,\mf p)&:=\frac{1}{(r-4)!}p_0^{r-4}p_1^4.
\end{align*}
The above definitions present a slight abuse of notation, as the terms involving $p_1,p_2,p_3$ and $p_4$ are special combinations of each $p_j\in\R^{d_j}$ giving scalar terms.

Then, we have 
\begin{equation}
    \label{verif-n1-bis}
    N_1(u,Du,D^2u,D^3u,D^4u) = \D^2\pare{\sum_{r\ge2}(-1)^r\frac{u^r}{r!}}.
\end{equation}
Therefore, in view of \eqref{verif-n1}, system \eqref{eq:18-1-bis} is equivalent (up to multiplying each term in $N_1$ by some suitable constant coefficients) to 
\[
\begin{cases}
\ds\pat u + \D^2u=N_1(u,Du,D^2u,D^3u,D^4u) &\qquad \mbox{ in }\q  \R_+\!\times\TT^d\,, \\
\ds u\bigl|_{t=0}\,=\,\D v_0 &\qquad \mbox{ in }\q\, \TT^d\,.
\end{cases}
\]
From \eqref{verif-n1-bis}, we immediately see that the non-linearity $N_1$ satisfies assumption \tbf{(A3)}. We now check that assumptions \tbf{(A1)} and \tbf{(A2)}
are also verified. Similarly to what done in Example \ref{example},
it is enough to estimate the non-linearity $N_1$ in terms of its variables $(p_0,\mf p)$.

Using that
\begin{align*}
\sum_{r\ge2} \frac{w^{r-1}}{(r-1)!}
=e^w-1\approx w  \qquad\text{and}\qquad
\sum_{r\ge j} \frac{w^{r-j}}{(r-j)!}
=e^w,
\end{align*}
we can estimate
\begin{align*}
 \av{ N_1(p_0,\mf p)}
&\lesssim \av{p_0}\av{p_4}+e^{\av{p_0}} \pare{ |p_2|^2+ \av{p_3}\av{p_1}+\av{p_2}\av{p_1}^2 +\av{p_1}^4}\\
&\lesssim \av{p_0}\av{p_4}+ |p_2|^2+ \av{p_3}^{4/3} +\av{p_1}^4.
\end{align*}
Hence, we have proved that $N_1$ verifies \tbf{(A1)} with $\ell_1=\ell_2=\ell_3=0$ and $\ell_4=1$.

We now differentiate the function $N_1$ and obtain the following estimates:
\begin{align*}
\av{\N_{p_0} N_1(p_0,\mf p)}&\lesssim \sum_{r\ge2} \frac{1}{(r-2)!}\av{p_0}^{r-2}\av{p_4}+ \sum_{r\ge3}\frac{1}{(r-3)!}\av{p_0}^{r-3}\left[\av{p_2}^2+\av{p_3}\av{p_1}\right]\\
&\q +\sum_{r\ge4}\frac{1}{(r-4)!}\av{p_0}^{r-4}\av{p_2}\av{p_1}^2+\sum_{r\ge5}\frac{1}{(r-5)!}\av{p_0}^{r-5}\av{p_1}^4\\
&\le e^{\av{p_0}}\bra{\av{p_4}+\av{p_2}^2+\av{p_3}\av{p_1}+\av{p_2}\av{p_1}^2+\av{p_1}^4}\\
&\lesssim \av{p_4}+\av{p_2}^2+\av{p_3}^{4/3} +\av{p_1}^4,
\\
\av{\N_{p_1} N_1(p_0,\mf p)}&\lesssim \sum_{r\ge2} \frac{1}{(r-2)!}\av{p_0}^{r-2}\av{p_3} \\
&\qquad\qquad +\sum_{r\ge3}\frac{1}{(r-3)!}\av{p_0}^{r-3}\av{p_2}\av{p_1}+\sum_{r\ge4}\frac{1}{(r-4)!}\av{p_0}^{r-4}\av{p_1}^3\\
&\le e^{\av{p_0}}\bra{\av{p_3}+\av{p_2}\av{p_1}+\av{p_1}^3}\\
&\lesssim  \av{p_3}+\av{p_2}^{3/2}+\av{p_1}^3,
\\
\av{\N_{p_2} N_1(p_0,\mf p)}&\lesssim \sum_{r\ge2} \frac{1}{(r-2)!}\av{p_0}^{r-2}\av{p_2}+\sum_{r\ge3}\frac{1}{(r-3)!}\av{p_0}^{r-3} \av{p_1}^2 \\
&\lesssim e^{\av{p_0}}\bra{\av{p_2}+\av{p_1}^2} \lesssim \av{p_2}+\av{p_1}^2,
\\
\av{\N_{p_3} N_1(p_0,\mf p)}&\lesssim\sum_{r\ge2} \frac{1}{(r-2)!}\av{p_0}^{r-2}\av{p_1}\le e^{\av{p_0}}\av{p_1} \lesssim \av{p_1},\\
\av{\N_{p_4} N_1(p_0,\mf p)}&\lesssim\sum_{r\ge2} \frac{1}{(r-1)!}\av{p_0}^{r-1}\le e^{\av{p_0}}-1\approx \av{p_0}.
\end{align*}
Thus, we see that  \tbf{(A2)} is verified with  $q_0=q_1=q_2=q_3=0$ and $q_4=1$.

\subsubsection{Verifying {\rm\tbf{(A1)}}, {\rm\tbf{(A2)}} and {\rm\tbf{(A3)}} for \eqref{eq:18-2}}

Using the change of variable $u=-1+1/v$, we get the following equivalent formulation of \eqref{eq:18-2}: 
\begin{equation}\label{eq:18-2-bis}
    \begin{cases}
\begin{array}{ll}
\ds\pat u\,=\,\D^2\pare{\frac{1}{(1+u)^3}}&\qquad \text{ in }\q  \R_+\!\times\,\TT^d\,, \\
\ds u\bigl|_{t=0}\,=\,-\,1\,+\,  \frac{1}{v_0} &\qquad \text{ in }\q\, \TT^d\,.
\end{array}
\end{cases}
\end{equation}
We can rewrite the right-hand side as  
\begin{equation}\label{verif-n2}
    \D^2\pare{\frac{1}{(1+u)^3}}=\D^2\pare{\sum_{r\ge0}(-1)^r\binom{r+2}{r}u^r}=-3\D^2u+\D^2\pare{\sum_{r\ge2}(-1)^r\binom{r+2}{r}u^r}.
\end{equation}
Let us now define the function 
\begin{align*}
&N_2(p_0,\mf p) := \sum_{s=1}^{4}\sum_{r\ge2}(-1)^r\tn_{s}^r(p_0,\mf p)  ,
\end{align*}
where
\begin{align*}
\tn_{1}^r(p_0,\mf p)  &:=(r+2)(r+1)rp_0^{r-1}p_4, \\
\tn_{2}^r(p_0,\mf p)  &:=(r+2)(r+1)r(r-1)p_0^{r-2}\left[p_2^2+p_3p_1\right], \\
\tn_{3}^r(p_0,\mf p)  &:=(r+2)(r+1)r(r-1)(r-2)p_0^{r-3}p_2p_1^2, \\
\tn_{4}^r(p_0,\mf p)  &:=(r+2)(r+1)r(r-1)(r-2)(r-3)p_0^{r-4}p_1^4.
\end{align*}
Then, we have 
\begin{equation}
    \label{verif-n2-bis}
    N_2(u,Du,D^2u,D^3u,D^4u) = \D^2\pare{\sum_{r\ge2}(-1)^r\binom{r+2}{r}u^r}.
\end{equation}
Therefore, in view of \eqref{verif-n2}, system \eqref{eq:18-2-bis} is equivalent to 
\[
\begin{cases}
\ds\pat u + 3\D^2u=N_2(u,Du,D^2u,D^3u,D^4u) &\qquad \mbox{ in }\q  \R_+\!\times\TT^d\,, \\
\ds u\bigl|_{t=0}\,=\,-1+\frac{1}{v_0} &\qquad \mbox{ in }\q\, \TT^d\,.
\end{cases}
\]
From \eqref{verif-n2-bis}, we immediately see that the non-linearity $N_2$ satisfies assumption \tbf{(A3)}. We now check that assumptions \tbf{(A1)} and \tbf{(A2)}
are also verified. Once again, it is enough to check the growth of $N_2$ and its derivatives with respect to the variables $(p_0,\mf p)$.
Using that
\begin{align*}
\sum_{r\ge2} (r+2)(r+1)r\,w^{r-1}&
=\frac{3!}{(1-w)^4}-3!\,,
\\
\sum_{r\ge k} (r+2)(r+1)\ldots(r-k+1)\,w^{r-k}&
=\frac{(k+2)!}{(1-w)^{k+3}}\quad\text{for}\quad k\ge2\,,
\end{align*}
we obtain the following estimates:
\begin{align*}
\av{\sum_{r \ge2} (-1)^r\tn_{1}^r(p_0,\mf p)}&\lesssim\pare{\frac{1}{(1-\av{p_0})^4}-1}\av{p_4}\lesssim\av{p_4} \av{p_0} ,
\\
\av{\sum_{r \ge2} (-1)^r\tn_{2}^r(p_0,\mf p)}&\lesssim \frac{1}{(1-\av{p_0})^5}\left[|p_2|^2+ \av{p_3}\av{p_1} \right]\lesssim |p_2|^2+ \av{p_3}^{4/3}+\av{p_1}^4,
\\
\av{\sum_{r \ge2} (-1)^r\tn_{3}^r(p_0,\mf p)}&\lesssim \frac{1}{(1-\av{p_0})^6}\av{p_2}\av{p_1}^2  \lesssim \av{p_2}^2+\av{p_1}^4,
\\
\av{\sum_{r \ge2} (-1)^r\tn_{4}^r(p_0,\mf p)}&\lesssim \av{p_1}^4.
\end{align*}
Then,  $N_2$ verifies \tbf{(A1)} with $\ell_1=\ell_2=\ell_3=0$ and $\ell_4=1$.

We now estimate the derivatives of $N_2$ in the $p_j$-variables and obtain:
\begin{align*}
\av{\N_{p_0} N_2(p_0,\mf p)}&\lesssim\sum_{r\ge2} \biggl[
(r+2)(r+1)r(r-1)\av{p_0}^{r-2}\av{p_4}\\
&\qq
+(r+2)(r+1)r(r-1)(r-2)\av{p_0}^{r-3}\left[\av{p_2}^2+\av{p_3}\av{p_1}\right]\biggr]\\
&\qq
+\sum_{r\ge3}(r+2)(r+1)r(r-1)(r-2)(r-3)\av{p_0}^{r-4}\av{p_2}\av{p_1}^2\\
&\qq
+\sum_{r\ge4}(r+2)(r+1)r(r-1)(r-2)(r-3)(r-4)\av{p_0}^{r-5}\av{p_1}^4
\\
&\lesssim
\frac{1}{(1-\av{p_0})^5}\av{p_4}
+\frac{1}{(1-\av{p_0})^6}\left[\av{p_2}^2+\av{p_3}\av{p_1}\right]
+\frac{1}{(1-\av{p_0})^7}\av{p_2}\av{p_1}^2\\
&\qq
+\frac{1}{(1-\av{p_0})^8}\av{p_1}^4
 \\
&\lesssim
\av{p_4}
+\av{p_2}^2+\av{p_3}^{4/3}+\av{p_1}^4
,
\\
\av{\N_{p_1} N_2(p_0,\mf p)}&\lesssim\sum_{r\ge2} \biggl[
(r+2)(r+1)r(r-1)\av{p_0}^{r-2}\av{p_3}\\
&\qq
+(r+2)(r+1)r(r-1)(r-2)\av{p_0}^{r-3}\av{p_2}\av{p_1}\biggr]\\
&\qq
+\sum_{r\ge3}(r+2)(r+1)r(r-1)(r-2)(r-3)\av{p_0}^{r-4}\av{p_1}^3
\\
&\le
\frac{1}{(1-\av{p_0})^5}\av{p_3}+\frac{1}{(1-\av{p_0})^6}\av{p_2}\av{p_1}+\frac{1}{(1-\av{p_0})^7}\av{p_1}^3
\\
&\lesssim
\av{p_3}+\av{p_2}^{3/2}+\av{p_1}^3
,
\\
\av{\N_{p_2} N_2(p_0,\mf p)}&\lesssim\sum_{r\ge2} \biggl[(r+2)(r+1)r(r-1)\av{p_0}^{r-2}\av{p_2}\\
&\qq
+(r+2)(r+1)r(r-1)(r-2)\av{p_0}^{r-3}\av{p_1}^2
\biggr]
\\
&\le
\frac{1}{(1-\av{p_0})^5}\av{p_2}+\frac{1}{(1-\av{p_0})^6}\av{p_1}^2 \lesssim \av{p_2}+\av{p_1}^2 
,
\\
\av{\N_{p_3} N_2(p_0,\mf p)}&\lesssim\sum_{r\ge2} (r+2)(r+1)r(r-1)\av{p_0}^{r-2}\av{p_1}\le \frac{1}{(1-\av{p_0})^5}\av{p_1} \lesssim \av{p_1},\\
\av{\N_{p_4} N_2(p_0,\mf p)}&\lesssim\sum_{r\ge2} (r+2)(r+1)r\av{p_0}^{r-1}\approx \pare{\frac{1}{(1-\av{p_0})^4}-1} \approx \frac{1}{(1-\av{p_0})^4}\av{p_0} \lesssim \av{p_0}.
\end{align*}
Then,  $N_2$ verifies \tbf{(A2)} with $q_0=q_1=q_2=q_3=0$ and $q_4=1$.

\subsection{Verifying the assumptions for the model from \citep{magliocca2024fourth}}\label{ss:two-op} 
The  physical model studied in \citep{magliocca2024fourth}   (and proposed in \citep{khenner2018modeling} when $d=1$)
describes controlled solid-state dewetting processes, and its generalisation to $d>1$ is given by
\begin{equation}\label{eq:24-0}
\begin{cases}
\begin{array}{ll}
\ds\pat v\,=\,-V M\,g_0\,\D^2v\,-\,
V M\,\D\big(G(v)\D u- G'(v) \big)&\qquad \text{ in }\q  \R_+\!\times\,\TT^d\,,  \\
\ds v\bigl|_{t=0}\,=\,v_0 &\qquad \text{ in }\q\, \TT^d\,,
\end{array}
\end{cases}
\end{equation}
\[
\text{ where }\qq
G(v) := -\frac{\tilde{c}_1n}{n+v}+\frac{\tilde{c}_2n^2}{(n+v)^2}\qq\text{and}\qq V,\,M,\,g_0,\,\tilde{c}_1,\,\tilde{c}_2,\,n\in\;]0,+\infty[.
\]

The existence of solutions $v>0$ with $\norm{v_0}_{A^0}<1$ and verifying $\norm{v(t)}_{L^\infty}\lesssim\norm{v_0}_{A^0}$ is proved in \cite[Theorem 2.2]{magliocca2024fourth}.

To begin with, let us remark that, since any solution $v$ to equation \eqref{eq:24-0} satisfies $\int_{\TT^d} \partial_t v\,\dx=0$, we have, for all times of existence $t$ of the solution $v$, the mass conservation 
\begin{equation}
    \label{massconv}
    \int_{\TT^d} v(t,x) \; dx = \int_{\TT^d} v_0(x)  \;\dx.
\end{equation}
Performing the adimensionalisation in \cite[Subsection 1.2]{magliocca2024fourth}, system \eqref{eq:24-0} can be equivalently written in the following form:
\begin{equation}\label{eq:24-0-bis}
\begin{cases}
\ds\pat u - \al_1\D u + \al_2\D^2u = N_1(u,Du,D^2u) + N_2(u,Du,\ldots,D^4u)  &\qquad \text{ in }\q  \R_+\!\times\TT^d\,,  \\
\ds u\bigl|_{t=0}\,=\,\frac{v_0-\|v_0\|_{L^1}}{n+\|v_0\|_{L^1}} &\qquad \text{ in }\q\, \TT^d\,,
\end{cases}
\end{equation}
where we have defined
\[
\al_1 := 2(3c_2-c_1) > 0\qq \mbox{ and }\qquad \al_2 :=1-c_1+c_2 > 0,
\]
for suitable constants $c_1,c_2>0$ (see \cite{magliocca2024fourth} for details), and the non-linearities are defined by 
\begin{align*}
N_1(p_0,\mf p_2)&:=\sum_{r\ge1} (-1)^r\,
{\tn}_{0}^r(p_0,\mf p_2)
+\sum_{r \ge0} (-1)^r {\tn}_{1}^r(p_0,\mf p_2) , \\
N_2(p_0,\mf p)&:=-\sum_{r\ge1} (-1)^r\,
{\tn}_{2}^r(p_0,\mf p)-
\sum_{r \ge0} (-1)^r\Big({\tn}_{3}^r(p_0,\mf p)+{\tn}_{4}^r(p_0,\mf p)\Big) ,
\end{align*}
with
\begin{align*}
\tn_{0}^r(p_0,\mf p_2)&:= [-\cu (r+2)(r+1)+\cd (r+3)(r+2)(r+1)]\,
p_0^r
p_2,
\\
\tn_{1}^r(p_0,\mf p_2)&:= [\cu (r+3)(r+2)(r+1)-\cd (r+4)(r+3)(r+2)(r+1)]\,
p_0^r
p_1^2,
\end{align*}
and
\begin{align*}
\tn_{2}^r(p_0,\mf p)&:=
[-\cu +(r+1)\cd ]\,p_0^rp_4,
\\
\tn_{3}^r(p_0,\mf p)&:=
[\cu (r+1)-\cd (r+2)(r+1)]\,p_0^r
\,(2p_1p_3+p_2^2),
\\
\tn_{4}^r(p_0,\mf p)&:=
[-\cu (r+2)(r+1)+\cd (r+3)(r+2)(r+1)]
\,p_0^r
\,p_1^2p_2.
\end{align*}
Note that the reformulation \eqref{eq:24-0-bis} holds up to multiplying each term appearing in the definitions of $N_1$ and $N_2$ by some suitable constant coefficient.

\medskip

Obviously, system \eqref{eq:24-0-bis} is not exactly of the form \eqref{eq:parab}, so our results do not directly apply in this case. However, provided that the non-linearities $N_1$ and $N_2$ satisfy suitable assumptions, the same kind of analysis performed in Sections \ref{s:proof-unique} and \ref{s:proof-exist} still holds. 

First, let us remark that, owing to the mass conservation \eqref{massconv} for the original system \eqref{eq:24-0}, such a property also holds for solutions to the adimensionalised problem \eqref{eq:24-0-bis}. Therefore, although it is not clear whether the non-linearities $N_1$ and $N_2$ satisfy respectively assumptions \tbf{(A3.$1$)} and \tbf{(A3)},
equation \eqref{eq:24-0-bis} does preserve the average of its solutions.
Next, it must be noted that both the terms $- \al_1\D u$ and $\al_2\D^2u$ give positive contributions to the estimates of the $X_{0,4}$
norm of the solution, and each of them must be used to treat the corresponding non-linear term, namely $N_1$ and $N_2$ respectively.
Thus, provided that $N_1$ and $N_2$ satisfy respectively assumptions \tbf{(A1.$1$)}--\tbf{(A2.$1$)} and \tbf{(A1)}--\tbf{(A2)},
the theory developed in the previous sections will straightforwardly apply. 

\medbreak
In order to simplify the calculations, we now set $c_1=c_2=1$. We are going to prove that:
\begin{itemize}[leftmargin=1.3em,itemsep=0.1em]
\item $N_1$ fulfills assumptions \tbf{(A1.$1$)} and \tbf{(A2.$1$)};
\item $N_2$ fulfills assumptions \tbf{(A1)} and \tbf{(A2)}.
\end{itemize}

We estimate
\begin{align*}
\av{\sum_{r \ge1} (-1)^r\tn_{0}^r(p_0,\mf p_2)}&\lesssim
\av{p_2} \av{\sum_{r \ge1}[-  (r+2)(r+1)+  (r+3)(r+2)(r+1)]
\av{p_0}^r}\\
&\lesssim \frac{1}{(1-\av{p_0})^4}\av{p_2}\av{p_0} \lesssim \av{p_2}\av{p_0},\\
\av{\sum_{r \ge0} (-1)^r\tn_{1}^r(p_0,\mf p_2)}&\lesssim\av{p_1}^2
 \av{\sum_{r \ge0}
[ (r+3)(r+2)(r+1)- (r+4)(r+3)(r+2)(r+1)]\av{p_0}^r
}\\
&\lesssim\frac{1+\av{p_0}}{(1-\av{p_0})^5}\av{p_1}^2 \lesssim \av{p_1}^2,   \\
\av{\sum_{r \ge1} (-1)^r\tn_{2}^r(p_0,\mf p)}&\lesssim \av{p_4}
 \av{\sum_{r \ge1} (r+1)\,\av{p_0}^r} = \frac{1}{(1-\av{p_0})^2}\av{p_4}\av{p_0} \lesssim \av{p_4}\av{p_0},\\
\av{\sum_{r \ge0} (-1)^r\tn_{3}^r(p_0,\mf p)}&\lesssim
 \pare{ \av{p_1}\av{p_3}+ \av{p_2}^2}  \av{\sum_{r \ge0}
[ (r+1)- (r+2)(r+1)] \av{p_0}^r
}\\
&=   \frac{1+\av{p_0}}{(1-\av{p_0})^3}\pare{ \av{p_1}^4+\av{p_3}^{4/3}+ \av{p_2}^2} \lesssim \av{p_1}^4+\av{p_3}^{4/3}+ \av{p_2}^2,\\
\av{\sum_{r \ge0} (-1)^r\tn_{4}^r(p_0,\mf p)}&\lesssim \av{p_1}^2\av{p_2}
\av{\sum_{r \ge0}
[- (r+2)(r+1)+ (r+3)(r+2)(r+1)]\av{p_0}^r
}\\
&\lesssim  \frac{1+\av{p_0}}{(1-\av{p_0})^4}\pare{\av{p_1}^4+\av{p_2}^2 } \lesssim \av{p_1}^4+\av{p_2}^2.
\end{align*}
Then, we easily deduce that the non-linearities $N_1$ and $N_2$ verify, respectively, \tbf{(A1.$1$)} with 
$\ell_1=0$ and $\ell_2=1$, and \tbf{(A1)} with 
$\ell_1=\ell_2=\ell_3=0$ and $\ell_4=1$.

We now need to find suitable estimates on the derivatives of both $N_1$ and $N_2$ in the $p_j$-variables: we compute
\begin{align*}
\av{\N_{p_0} N_1(p_0,\mf p_2)}&\lesssim\sum_{r\ge1} \biggl[
\bra{- (r+2)(r+1)r+ (r+3)(r+2)(r+1)r} \av{p_0}^{r-1} \av{p_2}\\
&\qq
\bra{(r+3)(r+2)(r+1)r- (r+4)(r+3)(r+2)(r+1)r}\av{p_0}^{r-1}\av{p_1}^2
\biggr]\\
&\lesssim  \frac{1+\av{p_0}}{(1-\av{p_0})^5}\av{p_2}+\frac{1+\av{p_0}}{(1-\av{p_0})^6}\av{p_1}^2 \lesssim \av{p_2}+\av{p_1}^2,
\\
\av{\N_{p_1} N_1(p_0,\mf p_2)}&\lesssim\sum_{r\ge0}  \bra{(r+3)(r+2)(r+1)- (r+4)(r+3)(r+2)(r+1)}\av{p_0}^r\av{p_1}\\
&\lesssim  \frac{1+\av{p_0}}{(1-\av{p_0})^5}\av{p_1} \lesssim \av{p_1}
,
\\
\av{\N_{p_2} N_1(p_0,\mf p_2)}&\lesssim\sum_{r\ge1} \bra{- (r+2)(r+1)+ (r+3)(r+2)(r+1)} \av{p_0}^r
\\
&\lesssim  \frac{\av{p_0}}{(1-\av{p_0})^4}  \lesssim \av{p_0}.
\end{align*}
In addiiton, we have
\begin{align*}
\av{\N_{p_0} N_2(p_0,\mf p)}&\lesssim\sum_{r\ge1}\biggl[   \bra{-r+(r+1)r}\av{p_0}^{r-1}\av{p_4}
\\
&\qq+
\bra{(r+1)r-(r+2)(r+1)r}\av{p_0}^{r-1}\left[\av{p_2}^2+\av{p_3}\av{p_1}\right]\\
&\qq
+\bra{-(r+2)(r+1)r+(r+3)(r+2)(r+1)r}\av{p_0}^{r-1}\av{p_2}\av{p_1}^2
\biggr]\\
&\lesssim \frac{1+\av{p_0}}{(1-\av{p_0})^3}\av{p_4}+\frac{1+\av{p_0}}{(1-\av{p_0})^4}\left[\av{p_2}^2+\av{p_3}\av{p_1}\right]+\frac{1+\av{p_0}}{(1-\av{p_0})^5}\av{p_2}\av{p_1}^2
 \\
&\lesssim  \frac{1+\av{p_0}}{(1-\av{p_0})^3}\av{p_4}+\frac{1+\av{p_0}}{(1-\av{p_0})^4}\left[\av{p_2}^2+\av{p_3}^{4/3}+\av{p_1}^4\right]\\
&\qq+\frac{1+\av{p_0}}{(1-\av{p_0})^5}\bra{\av{p_2}^2+\av{p_1}^4} \lesssim \av{p_4}+\av{p_2}^2+\av{p_3}^{4/3}+\av{p_1}^4
,
\\
\av{\N_{p_1} N_2(p_0,\mf p)}&\lesssim\sum_{r\ge0} \biggl[
[ (r+1)- (r+2)(r+1)]\,\av{p_0}^r\av{p_3}\\
&\qq+[- (r+2)(r+1)+ (r+3)(r+2)(r+1)]
\,\av{p_0}^r
\,\av{p_1} \av{p_2}
\biggr]\\
&=   \frac{1+\av{p_0}}{(1-\av{p_0})^3}\av{p_3}+\frac{1+\av{p_0}}{(1-\av{p_0})^4}\av{p_1} \av{p_2}
\\
&\lesssim  \frac{1+\av{p_0}}{(1-\av{p_0})^3}\av{p_3}+\frac{1+\av{p_0}}{(1-\av{p_0})^4}\bra{\av{p_1}^3+ \av{p_2}^{3/2}} \lesssim \av{p_3}+\av{p_1}^3+ \av{p_2}^{3/2}
,
\\
\av{\N_{p_2} N_2(p_0,\mf p)}&\lesssim\sum_{r\ge0} \biggl[[ (r+1)- (r+2)(r+1)]\,\av{p_0}^r
\,\av{p_2}\\
&\qq
+[- (r+2)(r+1)+ (r+3)(r+2)(r+1)]
\,\av{p_0}^r
\,\av{p_1}^2
\biggr]
\\
&= \frac{1+\av{p_0}}{(1-\av{p_0})^3}\av{p_2}+\frac{1+\av{p_0}}{(1-\av{p_0})^4}\av{p_1}^2 \lesssim \av{p_2}+\av{p_1}^2
,
\\
\av{\N_{p_3} N_2(p_0,\mf p)}&\lesssim\sum_{r\ge0} [ (r+1)- (r+2)(r+1)]\,\av{p_0}^r
\,\av{p_1}= \frac{1+\av{p_0}}{(1-\av{p_0})^3}\av{p_1} \lesssim \av{p_1},\\
\av{\N_{p_4} N_2(p_0,\mf p)}&\lesssim\sum_{r\ge1} [-1 +(r+1) ]\,\av{p_0}^r= \frac{\av{p_0}}{(1-\av{p_0})^2} \lesssim \av{p_0}.
\end{align*}
We then deduce that the non-linearities $N_1$ and $N_2$ verify, respectively, \tbf{(A2.$1$)} with 
$q_0=q_1=0$ and $q_2=1$, and \tbf{(A2)} with  
$q_0=q_1=q_2=q_3=0$ and $q_4=1$.

\subsection{The model considered in \citep{granero2024global}} \label{ss:appl-8}

The first model studied in \citep{granero2024global} related to the epitaxial growth is given by
\begin{equation}\label{eq:24-1}
\begin{cases}
\begin{array}{ll}
\ds\pat u
\,=\,
K_0\,\D u\,+\,2K_1\,\det D^2u\,-\,K_2\,\D^2u
 & \qquad \text{ in }\q  \R_+\!\times\,\TT^2\,, \\
\ds u\bigl|_{t=0}\,=\,u_0 &\qquad \text{ in }\q\, \TT^2\,,
\end{array}
\end{cases}
\end{equation}
with $K_0,K_1\ge0$ and $K_2>0$.
Its geometrical derivation can be found in \citep{escudero2023explicit,escudero2015global} and  \cite[Section 2]{escudero2013radial} (where the meaning of the constants $K_i$ is explained).
The existence of solutions to \eqref{eq:24-1} with $K_2-2K_1\norm{u_0}_{A^0}>1$ and verifying $\norm{u(t)}_{L^\infty}\lesssim\norm{u_0}_{A^0}$ is proved in \citep[Theorem 1.4]{granero2024global}.

Since $K_0\geq0$, similarly to what was done in Subsection \ref{ss:two-op}, let us reformulate system \eqref{eq:24-1} in the following form:
\[
\begin{cases}
\ds\pat u - K_0 \D u + K_2\,\D^2u
\,=\,
N(u,Du,D^2u,D^3u,D^4u)
 & \qquad \text{ in }\q  \R_+\!\times\,\TT^2\,, \\
\ds u\bigl|_{t=0}\,=\,u_0 &\qquad \text{ in }\q\, \TT^2\,,
\end{cases}
\]
where in this case the non-linearity is defined as
\[
N(u,Du,D^2u,D^3u,D^4u)=2K_1\det D^2u.
\]
Since $\det D^2u=u_{ii}u_{jj}-u_{ij}u_{ji}$, we see by an easy factorisation that the non-linearity $N$ satisfies assumption \tbf{(A3)},
and that 
\[
N(p_0,\mf p)\simeq p_2^2.
\]
Conditions \tbf{(A1)} and \tbf{(A2)} are then easily verified with $\ell_2=q_2=0$, since
\begin{align*}
\av{N(p_0,\mf p)}&\lesssim \av{p_2}^2,\\
\av{\N_{p_2}N(p_0,\mf p)}&\lesssim \av{p_2},\qq \big|\N_{p_j}N(p_0,\mf p)\big|=0\q \text{for}\q j\in\set{0,1,3,4}.
\end{align*}
Hence, Theorems \ref{teo:wp} and \ref{t:stab} also apply to this case.

\section*{Acknowledgements} 

R. G.-B. and M. M.  are funded by the project ”Análisis Matemático Aplicado y Ecuaciones Diferenciales” Grant
PID2022-141187NB-I00 funded by MCIN /AEI /10.13039/501100011033 / FEDER, UE and acronym ”AMAED”.
This publication is part of the project PID2022-141187NB-I00 funded by MCIN/ AEI /10.13039/501100011033.

M. M. thanks IMUS-Maria de Maeztu grant CEX2024-001517-M - Apoyo a Unidades de Excelencia María de Maeztu for supporting this research, funded by MICIU/AEI/ 10.13039/501100011033. The work of M. M. was funded by the Ministry of Science and Innovation/State Research
Agency/10.13039/501100011033 and by the European Union ”NextGenerationEU/Recovery, Transformation
and Resilience Plan”; the project PID2022-140494NA-I00 funded by MCIN/AEI/10.13039/50 \\
1100011033/FEDER, UE.

The work of F. F. has been partially supported by the project CRISIS (ANR-20-CE40-0020-01), operated by the French National Research Agency (ANR),
and by the project GHOST (PID2025-169177NB-I00), operated by the Spanish State Research Agency (AEI).
He also acknowledges the support of the Basque Government, through the BERC 2022-2025 program, of the Spanish State Research Agency (AEI),
through the BCAM Severo Ochoa excellence accreditation Grant CEX2021-001142-S funded by MICIU/AEI/10.13039/501100011033, and of
the European Union through the COFUND program [HORIZON-MSCA-2022-COFUND-101126600-SmartBRAIN3].

\section*{Statements and declarations}

\paragraph*{\bf Competing interests} The authors declare to have no competing interests.

\paragraph*{\bf Data availability statement} The authors declare that no sensible data are generated in this research.

\paragraph*{\bf Authors contribution} All authors have equally contributed to the research presented in this work, as well as to
its writing.

\end{document}